\documentclass[reqno]{amsart}
\usepackage{amsmath,amsthm,amssymb,enumitem,times,xcolor,bm,esint,mathtools,dsfont}
\usepackage{leftidx}
\usepackage[mathscr]{eucal}
\usepackage{harpoon}
\usepackage{todonotes}
\usepackage{tikz}
\usepackage[titletoc]{appendix} %附录编号
\usepackage[%backref=page,
linktocpage=true,colorlinks,citecolor=magenta,linkcolor=blue,urlcolor=magenta]{hyperref}
\allowdisplaybreaks[4]
\usepackage{etoolbox}
\patchcmd{\subsubsection}{\itshape}{\bfseries}{}{}
\usepackage{geometry}
\newtheorem{theorem}{Theorem}[section] 
\newtheorem{corollary}[theorem]{Corollary} 
 
\newtheorem{lemma}[theorem]{Lemma} 
\newtheorem{proposition}[theorem]{Proposition}
 
\theoremstyle{remark}
\newtheorem{remark}[theorem]{Remark} 
\numberwithin{equation}{section}  
\def\be{\begin{equation}}
\def\ee{\end{equation}}
\def\CR{{\mathcal C}}
\def\ER{{\mathcal E}}
\def\MR{{\mathcal M}}
\def\R{\mathbb R}
\def\S{\mathbb S}

\def\ga{\gamma}
\newcommand{\ve}{\varepsilon}
\def\iy{\infty}
\def\om{\omega}
\def\La{\Lambda}
\def\vp{\varphi}
\def\pa{\partial}
\def\ti{\tilde}
\def\f{\frac}
\def\tf{\tfrac}
\def\sbr#1{\left(#1\right)}

\def\abr#1{\left\langle#1\right\rangle}
\def\nm#1{\left\|#1\right\|}
\def\rd{\mathrm d}
\DeclareMathOperator{\dist}{dist}

\begin{document}
\theoremstyle{plain}

\title[Caffarelli--Kohn--Nirenberg inequality]{A lower bound for the stability constant in the Caffarelli--Kohn--Nirenberg inequality}

\author{Houwang Li}\address{School of Mathematics and Statistics, Beijing Institute of Technology, Beijing, China}\email{houwangli@bit.edu.cn}
\author{Juncheng Wei}\address{Department of Mathematics, The Chinese University of Hong Kong, Hong Kong, China}\email{wei@math.cuhk.edu.hk}

%\date{\today}
\keywords{}

\begin{abstract}
We study the dependence on the dimension of stability constants for the Caffarelli--Kohn--Nirenberg inequalities, using the weighted energy distance to the manifold of optimizers. In the subcritical non-degenerate case, i.e., $b_{FS}(a)<b<a+1$ with $a<0$ and $a<b<a+1$ with $a\ge0$, we prove
\[
 \frac{\beta_{a,b}}d\le c_{d,a,b} \le \frac{4(1+a-b)}{d+4(1+a-b)}
\]
with a positive constant depending only on the weights, uniformly over all admissible dimensions. The upper bound shows that this dimensional order is optimal. Our proof adapts the strategy of Dolbeault, Esteban, Figalli, Frank and Loss \cite{Dolbeault etc=2025} for Sobolev stability to the cylinder. It combines a local estimate on a neighborhood whose radius is independent of the dimension, a decomposition into three amplitude ranges, and a quantitative passage from local to global stability. The critical case $a=b>0$ exhibits a different dimensional scale,
\[
 \frac{\min\{\beta_{\mathrm{Sob}},4a\}}{25d^2}
 \le c_{d,a,a}\le\frac{16a}{d^2}+O_a(d^{-3}),
\]
where $\beta_{\mathrm{Sob}}>0$ is a universal Sobolev stability constant. The lower bound follows from an exact radial power transformation that separates the weighted deficit into a Sobolev deficit and an angular energy term. A geometric estimate controls the displacement of Sobolev optimizers from the origin, while an angular perturbation gives the upper bound. In the degenerate case on the Felli--Schneider curve, for fixed $a<0$, we prove
\[
 \frac{\beta_a^{FS}}{d^3}
 \le\widetilde c_{d,a,b_{FS}(a)}
 \le\frac{32|a|}{d^3}+O_a(d^{-4}),
\]
where $\widetilde c$ is the quartic stability constant given in \cite{Frank-Peteranderl=2024} and $\beta_a^{FS}>0$ is independent of $d$. A uniform chart of approximate conformal profiles, a truncation estimate retaining their Euler--Lagrange residual, and the local-to-global argument yield the lower bound. An explicit perturbation gives an upper bound in every dimension. The three cases thus have optimal dimensional orders $d^{-1}$, $d^{-2}$, and $d^{-3}$, respectively.
\end{abstract}

\maketitle
%%%%%%%----------------------------------------------------------------------

\section{Introduction}\label{sec:introduction}

The Caffarelli--Kohn--Nirenberg inequalities form a family of weighted interpolation inequalities in which the exponents and the powers of the distance to the origin are linked by scaling. Introduced in \cite{CKN=1984}, they provide a common framework for Sobolev and Hardy inequalities and for variational problems associated with singular or degenerate elliptic equations. The weights preserve dilations but generally destroy translation invariance. As a consequence, even the description of the optimizers involves questions of existence, concentration and symmetry breaking. In this paper we study a quantitative aspect of this theory: how the optimal stability constant depends on the dimension, with the weights kept fixed.

\subsection{The inequality, optimizers and symmetry}
We consider the Caffarelli--Kohn--Nirenberg (CKN) inequality
\be\label{eq:intro-ckn}
 \int_{\R^d}|x|^{-2a}|\nabla f|^2\rd x
 \ge S_{d,a,b}\left(\int_{\R^d}|x|^{-bq}|f|^q\rd x\right)^{2/q},
\ee
where $S_{d,a,b}$ is the sharp constant and
\be\label{eq:intro-parameters}
 d\ge2,\qquad -\iy<a<a_c:=\frac{d-2}{2},\qquad
 \begin{cases}
 a<b<a+1,&d=2,\\
 a\le b<a+1,&d\ge3,
 \end{cases}
 \qquad q:=\frac{2d}{d-2(1+a-b)}.
\ee
All functions are real-valued. The natural space is the weighted energy completion
\[
 D_a^{1,2}(\R^d):=\overline{C_c^\infty(\R^d\setminus\{0\})}^{\,\nm{\cdot}_{D_a^{1,2}}},\qquad
 \nm f_{D_a^{1,2}}^2:=\int_{\R^d}|x|^{-2a}|\nabla f|^2\rd x,
\]
with inner product
\[
 \langle f,g\rangle_{D_a^{1,2}}
 =\int_{\R^d}|x|^{-2a}\nabla f\cdot\nabla g\rd x.
\]
We use $D^{1,2}(\R^d)$ for the usual homogeneous Sobolev space. 
For $d\ge3$, the pair $(a,b)=(0,0)$ gives the sharp Sobolev inequality, while the endpoint $b=a+1$, excluded from \eqref{eq:intro-parameters}, gives the weighted Hardy inequality with constant $(a_c-a)^2$. We write
\[
 2^*:=\begin{cases}
 \displaystyle\frac{2d}{d-2},&d\ge3,\\
 +\infty,&d=2.
 \end{cases}
\]
The critical line $b=a$ corresponds to $q=2^*$; the range $a<b<a+1$ is subcritical.

The broader weighted interpolation framework was developed by Lin \cite{Lin=1986}. Early sharp-constant questions were studied by Chou and Chu \cite{Chou-Chu=1993}, while Horiuchi \cite{Horiuchi=1997} analyzed best constants and attainment for power weights. These results already exhibit the interaction between scaling, the admissible parameters and loss of compactness.
Catrina and Wang \cite{Catrina-Wang=2001} established existence of optimizers in the open range $a<b<a+1$ and exhibited symmetry breaking for negative $a$. They also identified the nonattainment phenomenon on $b=a<0$: in dimensions $d\ge3$, the sharp constant there equals the unweighted Sobolev constant but is not attained. 
These results distinguish loss of compactness from the separate question of whether an optimizer, when it exists, is radial.
The threshold for instability of the radial branch is described by the Felli--Schneider curve \cite{Felli-Schneider=2003}
\be\label{eq:intro-fs}
 b_{FS}(a):=a-a_c+\frac{d(a_c-a)}{2\sqrt{(a_c-a)^2+d-1}}.
\ee
For $a<0$ and $a<b<b_{FS}(a)$ the radial optimizer is unstable and symmetry is broken. Dolbeault, Esteban and Loss \cite{Dolbeault etc=2016} proved symmetry up to this threshold, using nonlinear diffusion and rigidity arguments. In particular, apart from the translation-invariant Sobolev case $(a,b)=(0,0)$, all optimizers in the region
\be\label{eq:intro-symmetry}
 \begin{cases}
 b_{FS}(a)\le b<a+1,&a<0,\\
 a\le b<a+1,&0\le a<a_c,
 \end{cases}
\ee
belong to
\[
 \MR_{a,b}:=\{cW_\tau:c\in\R,\ \tau>0\},\qquad
 W_\tau(x):=\tau^{a_c-a}W(\tau x),
\]
where
\be\label{W}
 W(x):=\bigl[2q(a_c-a)^2\bigr]^{1/(q-2)}
 \bigl[1+|x|^{(a_c-a)(q-2)}\bigr]^{-2/(q-2)}.
\ee
The zero function is included in $\MR_{a,b}$. For $(a,b)=(0,0)$, translations must also be included: the full manifold consists of all scalar multiples, dilations and translates of an Aubin--Talenti function \cite{Aubin=1976,Talenti=1976}. This distinction between the full Sobolev manifold and the centered CKN manifold is essential at the critical line.

The development of symmetry regions also includes \cite{DELT=2009,DEL=2012}, where compactness and spectral estimates yield qualitative symmetry results. For the larger CKN interpolation family and its weighted logarithmic Hardy limits, existence and rigidity were studied in \cite{Dolbeault-Esteban=2012,DEFT=2015}. These works provide background on optimizer structure; the present paper assumes symmetry and quantifies stability as the dimension varies.

\subsection{Stability and quantitative dependence on the dimension}
A stability inequality asks whether a small deficit forces a function to be close to the optimizer manifold, with an explicit bound in the energy norm. 
We write the energy norm and the deficit as
\be\label{eq:intro-deficit}
 \begin{aligned}
 E_a(f)&:=\nm f_{D_a^{1,2}}^2,\\
 D_{d,a,b}(f)&:=E_a(f)-S_{d,a,b}\left(\int_{\R^d}|x|^{-bq}|f|^q\rd x\right)^{2/q}.
 \end{aligned}
\ee
Following Wei and Wu \cite{Wei-Wu=2022,Wei-Wu=2024}, the \emph{non-degenerate case} is the parameter range
\be\label{case-1}
 b_{FS}(a)<b<a+1\quad\text{if }a<0;\qquad
 a\le b<a+1\quad\text{if }0\le a<a_c.
\ee
We exclude $(a,b)=(0,0)$ when using the centered manifold $\MR_{a,b}$ and define
\be\label{eq:intro-stability-constant}
 c_{d,a,b}:=\inf_{f\in D_a^{1,2}(\R^d)\setminus\MR_{a,b}}
 \frac{D_{d,a,b}(f)}{\dist_{D_a^{1,2}}(f,\MR_{a,b})^2}.
\ee
For $d\ge3$, Wei and Wu \cite{Wei-Wu=2022} established positivity of this constant in the non-degenerate case, giving a weighted analogue of the Bianchi--Egnell inequality. The two-dimensional counterpart was proved in \cite[Theorem 1.2]{Wei-Wu=2024}. The positivity statement for fixed dimension and weights does not itself provide a lower bound that can be tracked as $d$ increases.

The \emph{degenerate case} is the Felli--Schneider case
\be\label{case-2}
 a<0,\qquad b=b_{FS}(a).
\ee
In this case, the Hessian has additional zero modes that are not generated by the symmetries of the optimizer manifold. Frank and Peteranderl \cite{Frank-Peteranderl=2024} proved that the appropriate global remainder is of fourth order:
\be\label{eq:intro-degenerate-stability}
 \widetilde c_{d,a,b}:=
 \inf_{f\in D_a^{1,2}(\R^d)\setminus\MR_{a,b}}
 \frac{E_a(f)D_{d,a,b}(f)}{\dist_{D_a^{1,2}}(f,\MR_{a,b})^4}>0.
\ee
The exponent four is optimal in the degenerate case, and a quadratic-distance remainder cannot hold there. A different question concerns attainment of the best quadratic stability quotient in the non-degenerate case. Wei and Wu \cite{Wei-Wu=2024} proved existence of minimizers of \eqref{eq:intro-stability-constant} for certain parameters in the non-degenerate case, including those for which the lowest normal eigenvalue is on the radial branch $\mu_{0,2}$. More recently, Wei and Wu \cite{Wei-Wu=2026} studied nonnegative approximate critical points in the degenerate case and obtained sharp control of their distance to sums of bubbles by the one-third power of the Euler--Lagrange residual. This residual-based stability concerns the equation and multi-bubble configurations, whereas \eqref{eq:intro-stability-constant} measures the deficit of the functional inequality.

The search for quantitative Sobolev remainders goes back to Br\'ezis and Lieb \cite{Brezis-Lieb=1985}, who studied remainder terms and posed the optimizer-distance question. Lions's concentration-compactness principle for the critical case \cite{Lions=1985} supplies a general framework for the loss of compactness under dilations. Such compactness arguments give qualitative convergence of optimizing sequences, while dimensional bounds require additional quantitative estimates.

The unweighted Sobolev problem provides the model for the dimensional question. Let $S_d$ be the sharp Sobolev constant and let $\mathcal B$ denote the full manifold of Aubin--Talenti functions, including translations. Its centered submanifold is denoted by $\mathcal B_0$. Define the Sobolev deficit by
\be\label{eq:sobolev-deficit}
 D_{\mathrm{Sob}}(F):=\int_{\R^d}|\nabla F|^2\rd x-S_d\nm F_{2^*}^2.
\ee
Bianchi and Egnell \cite{Bianchi-Egnell=1991} proved that
\[
 D_{\mathrm{Sob}}(F)
 \ge C_{d,\mathrm{BE}}\dist_{D^{1,2}}(F,\mathcal B)^2
\]
with $C_{d,\mathrm{BE}}>0$. Their argument combines a local second-variation estimate with a compactness argument. Neither part of that proof gives a useful global lower bound with the correct behavior as $d\to\infty$.

Dolbeault, Esteban, Figalli, Frank and Loss \cite{Dolbeault etc=2025} overcame difficulties and proved that there is an explicit universal constant $\beta_{\mathrm{Sob}}>0$ such that
\be\label{eq:intro-sobolev-dimension}
 \frac{\beta_{\mathrm{Sob}}}{d}\le C_{d,\mathrm{BE}}\le\frac4{d+4},\qquad d\ge3.
\ee
Thus their result determines the optimal dimensional order $d^{-1}$, without identifying the exact best stability constant. The two main ingredients are a local estimate in a relative neighborhood of fixed radius, independent of $d$, and a quantitative local-to-global argument based on competing symmetries and continuous rearrangement. Their Sobolev estimate also yields a stability inequality for the Gaussian logarithmic Sobolev inequality with a dimension-independent constant by a large-dimensional limit. Explicit stability bounds and their dimensional dependence have also been studied for Hardy--Littlewood--Sobolev and fractional Sobolev inequalities by Chen, Lu and Tang \cite{Chen-Lu-Tang=2024,Chen-Lu-Tang=2025,Chen-Lu-Tang=2026}, using rearrangement, duality and refined local estimates. These developments motivate the corresponding question in the weighted setting, where the spectrum and the optimizer profile depend on the weights as well as the dimension.

Several related results address different aspects of stability. Chen, Frank and Weth \cite{Chen-Frank-Weth=2013} extended the Bianchi--Egnell remainder to fractional Sobolev inequalities, and Figalli and Zhang \cite{Figalli-Zhang=2022} obtained gradient stability with the optimal distance exponent for general first-order Sobolev inequalities. K\"onig \cite{Koenig=2023} showed that the global Sobolev stability constant is strictly below its local spectral value and proved attainment of the optimal quotient in \cite{Koenig=2025}. Thus the optimal dimensional order, the exact value of the constant and attainment of the stability quotient are distinct questions.

%On the dual side, Lieb's sharp Hardy--Littlewood--Sobolev theorem \cite{Lieb=1983} identifies the sharp inequality and optimizers in the conformally invariant case. Carlen \cite{Carlen=2017} developed a framework for transferring stability to dual inequalities. These results underlie the connection between Sobolev and Hardy--Littlewood--Sobolev stability used in the fractional theory cited above.

\subsection{Main results}
Our first result extends the dimensional lower-bound mechanism to the subcritical CKN inequalities in the non-degenerate case. The weights are fixed throughout the dimensional limit, and admissible dimensions are those satisfying \eqref{eq:intro-parameters} and the non-degeneracy assumptions.
Throughout the paper, we use the cylinder and its parameters
\be\label{eq:mass-parameter}
 \CR:=\R\times\S^{d-1},\qquad
 \La:=(a_c-a)^2,\qquad s:=\sqrt{1+\frac{d-1}{\La}}>1.
\ee

\begin{theorem}\label{thm:main}
Fix $a,b$ with $a<b<a+1$. There is a constant $\beta_{a,b}>0$, depending only on $a,b$, such that for every integer $d\ge2$ satisfying
\be\label{ourcase} b_{FS}(a)<b<a+1, \quad\text{if}~a<0;\quad a<b<a+1, \quad\text{if}~0\le a<a_c, \ee
one has
    \[c_{d,a,b}\ge\frac{\beta_{a,b}}d.\]
An explicit choice of $\beta_{a,b}$ is given in \eqref{eq:main-stability-constant} below.
\end{theorem}

The order in Theorem \ref{thm:main} agrees with a spectral upper bound:
\be\label{eq:intro-subcritical-upper}
 c_{d,a,b}\le\frac{2(q-2)}{3q-4}
 =\frac{4(1+a-b)}{d+4(1+a-b)}.
\ee
To see this, use the cylinder formulation of Section \ref{sec:cylinder} and take a nonzero radial eigenfunction $\psi$ corresponding to
\[
 \mu_{0,2}=\frac{(q-1)(3q-4)}q;
\]
compare \cite[Proof of Proposition 4.1]{Wei-Wu=2024}. It is orthogonal in the energy inner product to $V$ and $V'$, the tangent directions generated by amplitude and dilation. The second variation and the local normal-distance expansion give
\[
 \begin{aligned}
 \nm{V+\varepsilon\psi}_\La^2-S_{\La,q}\nm{V+\varepsilon\psi}_q^2
 &=\varepsilon^2\nm\psi_\La^2\left(1-\frac{q-1}{\mu_{0,2}}\right)+o(\varepsilon^2),\\
 \dist_\La(V+\varepsilon\psi,\MR)^2
 &=\varepsilon^2\nm\psi_\La^2+o(\varepsilon^2).
 \end{aligned}
\]
%Here the normal-distance expansion follows from the smooth amplitude and translation parametrization of $\MR$ near $V$. 
Taking $\varepsilon\to0$ proves \eqref{eq:intro-subcritical-upper}. Consequently, for each fixed pair of subcritical weights covered by Theorem \ref{thm:main}, the optimal dimensional order is $d^{-1}$.

The critical diagonal $a=b>0$ has a different dimensional scale. Its first angular mode, which becomes a translation mode in the unweighted Sobolev problem, produces a normal spectral gap of order $d^{-2}$. The next theorem gives matching lower and upper orders on this diagonal.

\begin{theorem}\label{thm:critical}
Fix $a>0$. For every integer $d>2+2a$, consider the critical case $b=a$, so that $q=2^*$, and set
\be\label{eq:critical-parameters}
 \alpha:=1-\frac{2a}{d-2}.
\ee
With $s$ as in \eqref{eq:mass-parameter}, one has
\be\label{eq:critical-two-sided}
\min\left\{\frac{\beta_{\mathrm{Sob}}}{25d^2},\frac{\alpha^{-2}-1}{24d+\alpha^{-2}-1}\right\}
\le c_{d,a,a}
\le\frac{(2s-q)(s+q-1)}{s(2s+q-2)}.
\ee
In particular, the constant
\be\label{eq:critical-beta}
\beta_a:=\frac1{25}\min\{\beta_{\mathrm{Sob}},4a\}>0
\ee
depends only on $a$ and satisfies $c_{d,a,a}\ge\beta_a/d^2$ in every such dimension. For fixed $a>0$,
\be\label{eq:critical-upper-asymptotic}
\frac{(2s-q)(s+q-1)}{s(2s+q-2)}
=\frac{16a}{d^2}+O_a(d^{-3})\qquad(d\to\infty).
\ee
Consequently, the order $d^{-2}$ is optimal.
\end{theorem}

Our third result concerns the degenerate case. Here $a<0$ is fixed, while $b$ follows the dimension-dependent Felli--Schneider curve. We occasionally write $b_{FS}^{(d)}(a)$ to display this dependence explicitly.

\begin{theorem}\label{thm:degenerate}
Fix $a<0$. For every integer $d\ge2$, take $b=b_{FS}^d(a)$ and set
\be\label{eq:fs-parameters}
 q=2s=2\sqrt{1+\frac{d-1}{\La}},\qquad n:=\frac2{q-2}.
\ee
There is a constant $\beta_a^{FS}>0$, depending only on $a$, such that
\be\label{eq:fs-two-sided}
 \frac{\beta_a^{FS}}{d^3}
 \le\widetilde c_{d,a,b_{FS}^d(a)}
 \le\frac{(2n+2-d)(2n+3)^2}
 {2n(n+1)(n+2)^2(d+2)}.
\ee
For fixed $a<0$, the upper bound satisfies
\be\label{eq:fs-upper-asymptotic}
 \frac{(2n+2-d)(2n+3)^2}
 {2n(n+1)(n+2)^2(d+2)}
 =\frac{32|a|}{d^3}+O_a(d^{-4})\qquad(d\to\infty).
\ee
Consequently, the optimal dimensional order of the quartic stability constant is $d^{-3}$.
\end{theorem}

The first two theorems are uniform in $d$ for fixed weights; Theorem~\ref{thm:degenerate} is uniform in $d$ for fixed $a$ with $b=b_{FS}^d(a)$ along the FS curve. The constants may depend on these fixed parameters. In particular, the subcritical estimate is not asserted to be uniform as the weights approach the Felli--Schneider curve or the critical diagonal. At $a=b=0$, Euclidean translations belong to the optimizer manifold and their tangent directions are removed in the Sobolev stability problem. For fixed $a=b>0$, translations are no longer optimizer symmetries, and the corresponding angular perturbations contribute to the normal distance. This explains why the $d^{-2}$ order is compatible with the $d^{-1}$ Sobolev result \eqref{eq:intro-sobolev-dimension}. On the FS curve these angular directions enter the kernel of the Hessian, so a quartic estimate is required. The upper coefficients $16a$ and $32|a|$ in Theorems~\ref{thm:critical} and~\ref{thm:degenerate} are not claimed to identify the limits of the respective rescaled global constants.

\subsection{Comparison with Sobolev stability and strategy of proof}
The proof of Theorem \ref{thm:main} follows the local-to-global strategy of \cite{Dolbeault etc=2025}. In their Sobolev analysis, stereographic projection sends an optimizer to a constant on $\S^d$, and spherical harmonics diagonalize the second variation. For CKN, the Emden--Fowler transformation leads to the cylinder $\CR$ in \eqref{eq:mass-parameter} and a nonconstant axial optimizer $V$. The generalized eigenvalue problem
\[
 (-\pa_t^2-\Delta_\om+\La)v=\mu V^{q-2}v
\]
contains both angular and radial indices. To make this terminology precise, let $\Theta_{i,l}$ be a real orthonormal basis of spherical harmonics on $\S^{d-1}$, with
\be\label{equ-eigen-S}
 -\Delta_\om\Theta_{i,l}=i(i+d-2)\Theta_{i,l},
 \qquad i\ge0,\quad 1\le l\le l_{i,d},
\ee
where $l_{i,d}$ is the dimension of the degree-$i$ harmonic space. Every $v\in L^2(\CR)$ has the angular decomposition
\be\label{fourier}\begin{aligned}
 v(t,\om)&=\sum_{i=0}^{\infty}\sum_{l=1}^{l_{i,d}}
                   \phi_{i,l}(t)\Theta_{i,l}(\om),\quad\text{with}\\
 \phi_{i,l}(t)&=\int_{\S^{d-1}}v(t,\om)\Theta_{i,l}(\om)\rd\om.
\end{aligned}\ee
The index $i$ records the angular degree, while $l$ distinguishes harmonics of the same degree. The degree-zero harmonic is constant, so a radial function on the cylinder is independent of $\om$ and depends only on $t$. This corresponds to radial symmetry in the original Euclidean variables. For each fixed $i$, the one-dimensional eigenvalue problem for the axial coefficient has an index $j=0,1,\ldots$. Thus the separated eigenmodes have the form
\[
 v_{i,j,l}(t,\om)=f_{i,j}(t)\Theta_{i,l}(\om),
 \qquad\text{with generalized eigenvalue }\mu_{i,j};
\]
see \eqref{mu-ij}. We call $j$ the radial, or axial, index. The coefficients $f_{i,j}$ depend only on $t$, and the complete modes are radial precisely when $i=0$.

The first two radial branches, $(i,j)=(0,0)$ and $(0,1)$, are represented by $V$ and $V'$, with
\[
 \mu_{0,0}=1,\qquad\mu_{0,1}=q-1.
\]
They generate amplitude changes and axial translations of the optimizer. After these directions have been removed, the first radial branch is $(0,2)$. The first angular branch means $(1,0)$: the degree-one harmonic space is spanned by $\om_1,\ldots,\om_d$, and this branch consists of the $d$ modes $f_{1,0}(t)\om_k$, $1\le k\le d$. In the non-degenerate case, the first normal branch is therefore determined by
\[
 \mu_3=\min\{\mu_{0,2},\mu_{1,0}\}>q-1.
\]
Here normal directions are orthogonal to $\operatorname{span}\{V,V'\}$ in the energy inner product, and their smallest quadratic coercivity factor is $1-(q-1)/\mu_3$. This is useful for the spectral gap control.

%\paragraph{Local analysis.}
As in \cite{Dolbeault etc=2025}, a direct Taylor expansion is replaced by a decomposition of the relative perturbation $r/V$ into three amplitude ranges. The bounded part is controlled by the spectral gap, the intermediate part has small support in the measure with density proportional to $V^q$, and the large tail is controlled by the sharp inequality itself. The estimates are combined so that the remainder can be absorbed at the coercivity scale $q-2$. Lemma \ref{lem:bound} gives a positive lower bound, independent of $d$ at fixed $a,b$, for
\[
 \frac{\mu_3-(q-1)}{q-2}.
\]
This replaces the explicit spherical spectral-gap estimate in the Sobolev argument.
A further ingredient is the weighted fourth-moment estimate in Lemma \ref{lemma:large eigenfunction}. It applies to every eigenfunction in a complete eigenspace below the cutoff $\mu_{K,0}$ and gives
\[
 \left(\int_\CR V^{q-4}v^4\rd\nu\right)
 \left(\int_\CR V^q\rd\nu\right)
 \le D_{K,a,b}\left(\int_\CR V^{q-2}v^2\rd\nu\right)^2.
\]
The proof combines radial hypergeometric identities with a reverse H\"older estimate for spherical harmonics, while grouping their full angular multiplicities. Together with a dimension-independent bound on the number of distinct eigenvalues below the cutoff, this yields the projection estimate for the intermediate truncation. These cylinder estimates account for the nonconstant weight in place of the spherical-harmonic estimates used in \cite{Dolbeault etc=2025}. They give Theorem \ref{thm:local} on a neighborhood whose radius is independent of $d$.

Rearrangement and diffusion provide two further sources of constructive estimates. Carlen and Loss \cite{Carlen-Loss=1990} used competing symmetries to construct optimizing sequences. Brock \cite{Brock=1995} introduced continuous Steiner symmetrization, and Brock and Solynin \cite{Brock-Solynin=2000} developed approximation by polarizations. These constructions form part of the background for the continuous-rearrangement strategy in \cite{Dolbeault etc=2025}.
Relative entropy and optimization over scales yield explicit stability estimates in \cite{Dolbeault-Toscani=2016}. Bonforte, Dolbeault, Nazaret and Simonov \cite{Bonforte-Dolbeault-Nazaret-Simonov=2025} developed constructive Gagliardo--Nirenberg--Sobolev stability through fast diffusion, regularity and entropy production. Their distances and normalization assumptions differ from the weighted energy distance used here; the cylinder argument below uses an energy-controlled path without requiring a quantitative convergence rate.

%\paragraph{From local to global stability.}
The second part retains the continuous-rearrangement principle of \cite{Dolbeault etc=2025}. On the cylinder, fiberwise axial Steiner symmetrization is combined with the normalized inverse map defined in \eqref{eq:inverse-map}.
The inverse map supplies regularization and an energy-controlled interpolation. Lemma \ref{lem:energy-path} constructs an $H^1$-continuous path of nonnegative, $L^q$-normalized functions from the initial function to an optimizer, with energy bounded above by its initial value. Subcritical compactness in an even, axially decreasing cone and rigidity of positive critical points \cite{Dolbeault etc=2016} identify the endpoint. Quantitative convergence rates are not needed. Crossing the boundary of the relative local neighborhood transfers the local lower bound to a global one without an additional dimension-dependent loss. Finally, a positive/negative-part argument, following \cite{Dolbeault etc=2025}, gives the estimate for arbitrary real-valued functions. The cylinder construction accounts for the unbounded axial variable and the nonconstant optimizer.

\vskip0.1in
\paragraph{\bf The critical case.}
The proof of Theorem \ref{thm:critical} uses the Sobolev stability theorem of \cite{Dolbeault etc=2025} through a direct geometric comparison. 
With $\alpha$ from \eqref{eq:critical-parameters}, we use the $b=a$ specialization of the radial power change in \cite[(3.1)]{Dolbeault etc=2016}:
\[
 y=|x|^{\alpha-1}x,\qquad f(x)=F(y).
\]
Define the angular energy by
\be\label{eq:angular-energy}
 A(F):=\int_{\R^d}\left|\nabla_y F-\frac{y}{|y|}\pa_rF\right|^2\rd y.
\ee
Then
the critical weighted deficit splits exactly as
\[
 D_{d,a,a}(f)=\alpha\bigl(D_{\mathrm{Sob}}(F)+(\alpha^{-2}-1)A(F)\bigr).
\]
The CKN optimizer manifold becomes the centered Sobolev manifold $\mathcal B_0$, and its weighted squared distance becomes
\[
 \alpha\bigl(\dist_{D^{1,2}}(F,\mathcal B_0)^2+(\alpha^{-2}-1)A(F)\bigr).
\]
The additional geometric estimate in Lemma \ref{lem:critical-center-control} bounds the squared distance to $\mathcal B_0$ by $25d$ times the squared distance to the full Sobolev manifold and $24d$ times the angular energy. It is this centering step that introduces an additional factor $d$ into the lower bound. The Sobolev coefficient $\beta_{\mathrm{Sob}}/d$ consequently contributes $\beta_{\mathrm{Sob}}/(25d^2)$. The angular contribution has the same scale, since
\[
 \alpha^{-2}-1=\frac{4a}{d}+O_a(d^{-2}),\qquad
 \frac{\alpha^{-2}-1}{24d+\alpha^{-2}-1}
 =\frac{a}{6d^2}+O_a(d^{-3}).
\]
Thus both terms in the lower bound \eqref{eq:critical-two-sided} have order $d^{-2}$. 

\vskip0.1in
\paragraph{\bf The degenerate case.}
We use \emph{zero mode} to denote the kernel of the Hessian of the deficit $D(u)$. Homogeneity and axial-translation invariance give the trivial zero modes $V$ and $V'$.
In the non-degenerate case, the third eigenvalue is larger than $q-1$, so that there are no other zero modes, which induces the orthogonality conditions \eqref{con-2}.
While on the FS curve, $\mu_{1,0}=q-1=\mu_{0,1}$, and the first angular modes become the additional zero modes
\[
 V^{q/2}(t)\om_k,\qquad 1\le k\le d;
\]
see \cite[(3.1)]{Frank-Peteranderl=2024}. 
These are called nontrivial zero modes because they do not arise from symmetries of the CKN optimizer manifold. 
The projection onto these additional modes and the resulting quartic analysis are the starting point of \cite[Propositions~3--4]{Frank-Peteranderl=2024}. To prove Theorem~\ref{thm:degenerate} with constants uniform in $d$, we represent them by the normalized approximate conformal profiles $\phi_{\sigma,\tau}$ of Section~\ref{sec:degenerate}. A modulation chart with a dimension-independent radius separates the profile parameters from a perturbation orthogonal to the full tangent space. In the non-degenerate subcritical case only amplitude and axial translation are removed, and the remaining directions have a positive quadratic gap; on the FS curve the angular kernel must also be included in this modulation.

The three-range truncation and the weighted low-mode fourth-moment estimates still control the transverse perturbation at scale $q-2\asymp d^{-1}$. 
The profile $\phi_{\sigma,\tau}$ is an approximate critical point, so its Euler--Lagrange residual must be retained. In the normalization of Section~\ref{sec:degenerate}, Lemma~\ref{lem:fs-robust} gives, when $\phi_{0,\tau}+h\ge0$, $h\perp_{\La,\nu_1}\mathcal T_{0,\tau}$, and $\|h\|_{q,\nu_1}$ is below a fixed dimension-independent threshold,
\[
 \mathsf D(\phi_{0,\tau}+h)
 \ge\mathsf D(\phi_{0,\tau})+c_a(q-2)\mathsf E(h)
 -\frac{C_a}{q-2}
       \|\mathcal A\phi_{0,\tau}-\phi_{0,\tau}^{q-1}\|_*^2.
\]
This is the counterpart of the subcritical local estimate, with a separate profile deficit controlling the zero-mode parameters.
The order $d^{-3}$ comes from that profile deficit $\mathsf D(\phi_{0,\tau})$. 
The logarithmic energy expansion \eqref{eq:fs-profile-log-energy} reads
\[
 \begin{aligned}
 \log\mathsf E(\phi_{0,\tau})
 ={}&\frac{n+1-d/2}{n+1}
 \left[-\log\left(1-\frac{|\tau|^2}d\right)
                       -\frac{|\tau|^2}d\right]
 +O_a\left(\frac{|\tau|^4}{d^4}\right).
 \end{aligned}
\]
Here $n\sim d/2$ and $n+1-d/2\to-2a>0$. The coefficient outside the brackets has order $d^{-1}$, while cancellation of the linear term inside the brackets leaves a term of order $|\tau|^4/d^2$. Since $\|\phi_{0,\tau}\|_{q,\nu_1}=1$, it follows that
\[
 \mathsf D(\phi_{0,\tau})
 =\mathsf E(\phi_{0,\tau})-1
 \asymp_a\frac{|\tau|^4}{d^3}.
\]
The three inverse powers are therefore already present in the fourth-order profile energy.
The residual estimate in Lemma~\ref{lem:fs-profile} gives
\[
 \frac1{q-2}
 \|\mathcal A\phi_{0,\tau}-\phi_{0,\tau}^{q-1}\|_*^2
 \le C_a\frac{|\tau|^4}{d^5}.
\]
This error can be absorbed into the profile deficit $\mathsf D(\phi_{0,\tau})$ for sufficiently large $d$. The remaining transverse term has order $d^{-1}\mathsf E(h)$ and hence controls $d^{-3}\mathsf E(h)^2$ on the fixed neighborhood. The chart then controls the fourth power of the distance by $C_a(|\tau|^4+\mathsf E(h)^2)$, yielding the local quartic estimate of Lemma~\ref{lem:fs-local}.

As in the non-degenerate subcritical proof, an energy-controlled path transfers the local estimate to a global one, and a positive/negative-part argument handles real-valued functions. At the FS boundary, endpoint rigidity identifies the limit of the path. The local-to-global factor $\delta_a^2$ is independent of $d$, while the sign-splitting coefficient $2(1-2^{2/q-1})\sim4\log2/d$ is larger in order than $d^{-3}$. These steps therefore preserve the dimensional order. Finally, an explicit approximate translation has computable hypergeometric energy and $L^q$ integrals; its fourth-order expansion gives the upper bound in \eqref{eq:fs-two-sided}. The non-degenerate subcritical and degenerate arguments share the truncation and global mechanisms, while the nonlinear treatment of the angular kernel produces the quartic $d^{-3}$ scale.

\vskip0.1in
\paragraph{\bf Organization of the paper.}
Section \ref{sec:cylinder} develops the cylinder formulation, the spectral gap and the weighted fourth-moment estimates. Section \ref{sec:local} proves the local stability estimate with a dimension-independent radius. Section \ref{sec:global} carries out the passage to global stability and proves Theorem \ref{thm:main}. Section \ref{sec:critical} proves the explicit critical bounds and Theorem \ref{thm:critical}. Section~\ref{sec:degenerate} proves the quartic bounds of Theorem~\ref{thm:degenerate} on the FS curve.

\section{The cylinder setting}\label{sec:cylinder}
\subsection{The CKN inequality on the cylinder}
Throughout this section, assume \eqref{case-1} and $(a,b)\ne(0,0)$.
Following \cite{Dolbeault etc=2016}, we reformulate the CKN inequality on the cylinder $\CR$ introduced in \eqref{eq:mass-parameter}.
With $\La$ as in \eqref{eq:mass-parameter}, condition \eqref{case-1} is equivalent to
\be\label{ourcase-1}
 2<q\le2^*,\qquad 0<\La<\La_{FS}:=\frac{4(d-1)}{q^2-4}.
\ee
Write $(t,\om)$ for the coordinates on $\CR$, and equip the cylinder with the product measure
\[
 \rd\nu=\rd t\,\rd\om,
\]
where $\rd\om$ is the usual surface measure on $\S^{d-1}$. Define
\[
 \abr{u,v}_\La:=\int_\CR\bigl(\pa_tu\pa_tv+\nabla_\om u\cdot\nabla_\om v+\La uv\bigr)\rd\nu,
 \qquad \nm u_\La^2:=\abr{u,u}_\La.
\]
For every fixed $\La>0$, this norm is equivalent to the usual $H^1(\CR)$ norm.

Under the Emden--Fowler transformation
\[
 x\mapsto(t,\om)=\left(\log|x|,\frac{x}{|x|}\right),
\]
a function $f\in D_a^{1,2}(\R^d)$ corresponds to
\[
 u(t,\om)=e^{(a_c-a)t}f(e^t\om)\in H^1(\CR).
\]
The transformation preserves the following quantities:
\be\label{equiv}\begin{aligned}
 \int_{\R^d}|x|^{-2a}|\nabla f|^2\rd x&=\nm u_\La^2,\\
 \int_{\R^d}|x|^{-bq}|f|^q\rd x&=\nm u_q^q:=\int_\CR|u|^q\rd\nu,\\
 \dist_{D_a^{1,2}}(f,\MR_{a,b})&=\dist_\La(u,\MR):=\inf_{v\in\MR}\nm{u-v}_\La,
\end{aligned}\ee
where
\be\label{MR}
 \MR=\{cV(\cdot-\tau):c\in\R,\ \tau\in\R\},
\ee
and
\be\label{V}
 V(t)=\left(\frac{q\La}{2}\right)^{1/(q-2)}
 \left[\cosh\left(\frac{q-2}{2}\sqrt\La\,t\right)\right]^{-2/(q-2)}.
\ee
The profile $V$ is the image of $W$ in \eqref{W}. By \eqref{equiv}, the sharp inequality constant is
\be\label{CKN-C-S}
 S_{\La,q}=\inf_{u\in H^1(\CR)\setminus\{0\}}\frac{\nm u_\La^2}{\nm u_q^2}=S_{d,a,b},
\ee
and the optimal stability constant satisfies
\be\label{CKN-C-c}
 c_{\La,q}=\inf_{u\in H^1(\CR)\setminus\MR}\ER(u)=c_{d,a,b},
\ee
where
\be\label{ER}
 \ER(u)=\frac{\nm u_\La^2-S_{\La,q}\nm u_q^2}{\dist_\La(u,\MR)^2}.
\ee
The optimizer manifold is $\MR$. Its normalized profile $V$, and every translate of $V$, satisfy
\be\label{equ-C}
 -\pa_t^2V-\Delta_\om V+\La V=V^{q-1}\qquad\text{on }\CR.
\ee

\subsection{The eigenvalue problem at an optimizer}
For the profile $V$ in \eqref{V}, consider
\be\label{equ-eigen}
 -\pa_t^2v-\Delta_\om v+\La v=\mu V^{q-2}v,\qquad v\in H^1(\CR).
\ee
The coefficient $q-1$ from the linearized equation is absorbed into $\mu$. Wei and Wu \cite[Proposition 3.1]{Wei-Wu=2024} studied this problem. We refine the estimates needed for the dimensional bounds. List the eigenvalues in increasing order:
\[
 0<\mu_1\le\mu_2\le\mu_3\le\cdots.
\]
Use the orthonormal spherical harmonics $\Theta_{i,l}$ from \eqref{equ-eigen-S} and the decomposition \eqref{fourier} introduced in Section~\ref{sec:introduction}.
By \eqref{equ-eigen} and \eqref{equ-eigen-S}, this expansion solves the eigenvalue problem if and only if each radial coefficient satisfies
\be\label{equ-eigen-R}\begin{aligned}
 -\pa_t^2\phi_{i,l}
 -\frac{\mu q\La}{2}\cosh^{-2}\left(\frac{q-2}{2}\sqrt\La\,t\right)\phi_{i,l}
 =-\bigl(\La+i(i+d-2)\bigr)\phi_{i,l}
 \qquad\text{on }\R.
\end{aligned}\ee
The negative eigenvalues of the one-dimensional operator on the left are
\[
 \sigma_j(\mu)=-\frac{\La(q-2)^2}{16}
 \left[\sqrt{1+\frac{8q\mu}{(q-2)^2}}-(1+2j)\right]^2,
\]
for the nonnegative integers $j$ satisfying
\[
 2j<\sqrt{1+\frac{8q\mu}{(q-2)^2}}-1;
\]
see \cite[p.~74]{Landau-Lifshitz=1958}, \cite[Section 4.2.2]{Frank etc=2022}, or \cite[p.~130]{Felli-Schneider=2003}. The strict inequality excludes the zero-energy threshold, which does not give an $L^2$ eigenfunction. A corresponding eigenfunction is
\[
\begin{aligned}
 \vp_j(\mu;t)
 ={}&\left[1-\tanh^2\left(\frac{q-2}{2}\sqrt\La\,t\right)\right]^{\frac{\sqrt{|\sigma_j(\mu)|}}{(q-2)\sqrt\La}}\\
 &\qquad \times\leftidx{_2}F_1\left(-j,\dfrac{4\sqrt{|\sigma_j(\mu)|}}{(q-2)\sqrt\La}+j+1;\dfrac{2\sqrt{|\sigma_j(\mu)|}}{(q-2)\sqrt\La}+1 ;\,\frac{1-\tanh\left(\frac{q-2}{2}\sqrt\La\,t\right)}2\right).
\end{aligned}
\]
Each negative eigenvalue is simple by one-dimensional Sturm--Liouville theory. Here $\leftidx{_2}F_1$ denotes the Gauss hypergeometric function, see \cite[Chapter 2]{book-1}. For $|z|<1$ and $d_3\notin\{0,-1,-2,\ldots\}$, it is defined by
\be\label{2-17}
 \leftidx{_2}F_1(d_1,d_2;d_3;z)
 :=\sum_{n=0}^{\infty}\frac{(d_1)_n(d_2)_n}{(d_3)_n\,n!}z^n,
\ee
where
\[
 (\xi)_n:=\begin{cases}
 1,&n=0,\\
 \xi(\xi+1)\cdots(\xi+n-1),&n\ge1.
 \end{cases}
\]
In particular, $(1)_n=n!$, and
\[
 (\xi)_n=\frac{\Gamma(\xi+n)}{\Gamma(\xi)}
 \qquad\text{if }\xi\notin\{0,-1,-2,\ldots\}.
\]
For the nonpositive integer first parameters used below, the hypergeometric series terminates and defines a polynomial.

Thus \eqref{equ-eigen-R} has a nonzero $H^1(\R)$ solution precisely when
\be\label{10-22-1}
 \La+i(i+d-2)=|\sigma_j(\mu)|
 =\frac{\La(q-2)^2}{16}
 \left[\sqrt{1+\frac{8q\mu}{(q-2)^2}}-(1+2j)\right]^2
\ee
for an admissible $j$. Solving for $\mu$ gives
\be\label{mu-ij}
 \mu_{i,j}=\frac{(q-2)^2}{8q}
 \left[\left(\frac{4\sqrt{\La+i(i+d-2)}}{(a_c-a)(q-2)}+1+2j\right)^2-1\right],
 \qquad i,j\ge0.
\ee
These are all the eigenvalues of \eqref{equ-eigen}.

\subsubsection{The first three eigenvalues and eigenfunctions}
Direct calculation gives
\[
 \mu_{0,0}=1,\qquad \mu_{0,1}=q-1,\qquad
 \mu_{0,2}=\frac{(q-1)(3q-4)}q,
\]
and
\[
 \mu_{1,0}=\frac2q+\frac{2(d-1)}{q\La}+\frac{q-2}{q}\sqrt{1+\frac{d-1}{\La}}.
\]
Condition \eqref{case-1} implies $\mu_{1,0}>\mu_{0,1}$. Moreover,
\[
 \mu_{i,j}<\mu_{i+1,j},\qquad \mu_{i,j}<\mu_{i,j+1}\qquad(i,j\ge0).
\]
The first two eigenvalues and their eigenspaces are therefore
\[
 \mu_1=1,\qquad \mu_2=q-1,\qquad
 E_{\mu_1}=\operatorname{span}\{V\},\qquad
 E_{\mu_2}=\operatorname{span}\{V'\},
\]
and
\[
 \mu_3=\min\{\mu_{0,2},\mu_{1,0}\}>\mu_2.
\]
Recall $s$ from \eqref{eq:mass-parameter}. Following \cite{Wei-Wu=2024}, define
\[
 a_c^*:=a_c\left(1-\sqrt{\frac{d-1}{2d}}\right),\qquad
 b_{\mathrm{FS}}^*(a):=a-a_c+\frac{d}{1+s}.
\]
\begin{lemma}\label{lem:third-eigenvalue}
Under the standing assumptions, let $\mu_3$ be the third eigenvalue.
Then
\[\mu_3=\begin{cases} 
\displaystyle\frac{(q-1)(3q-4)}q, & a_c^*\le a<a_c, \\[2mm]
\displaystyle\frac{(q-1)(3q-4)}q, & a<a_c^*,\quad b\ge b_{\mathrm{FS}}^*(a), \\[2mm] 
\displaystyle \frac2q+\frac{2(d-1)}{q\Lambda}+\frac{q-2}{q}\sqrt{1+\frac{d-1}{\Lambda}}, & a<a_c^*,\quad b<b_{\mathrm{FS}}^*(a).
\end{cases}\]
Each case is understood subject to the admissibility
conditions above. In particular, the first case is
empty when $d=2$.
\end{lemma}
\begin{proof}
A direct calculation yields
\[\begin{aligned}
    \mu_{1,0}-\mu_{0,2}
    &=\frac{2s^2+(q-2)s-(q-1)(3q-4)}q\\
    &=\frac{(s-q+1)(2s+3q-4)}q.
\end{aligned}\]
Since $2s+3q-4>0$, it follows that
    \[\mu_{0,2}\le\mu_{1,0}\quad\Longleftrightarrow\quad q\le1+s.\]
Using
    \[q=\frac{d}{a_c-a+b},\]
we obtain
    \[q\le1+s\quad\Longleftrightarrow\quad b\ge a-a_c+\frac{d}{1+s}=b_{\mathrm{FS}}^*(a).\]
For $d\ge3$, we further have
\[\begin{aligned}
    b_{\mathrm{FS}}^*(a)\le a \quad\Longleftrightarrow\quad  a\ge a_c\left(1-\sqrt{\frac{d-1}{2d}}\right)=a_c^*.
\end{aligned}\]
Consequently, if $a\ge a_c^*$, every admissible $b\ge a$ satisfies $b\ge b_{\mathrm{FS}}^*(a)$, and hence $\mu_3=\mu_{0,2}$. If $a<a_c^*$, the two branches are distinguished by whether $b\ge b_{\mathrm{FS}}^*(a)$ or $b<b_{\mathrm{FS}}^*(a)$.

Finally, when $d=2$, we have $a_c=a_c^*=0$ and $a<0$,
so the first case is empty. Moreover,
    \[a<b_{\mathrm{FS}}^*(a)=a+\frac2{1+s}<a+1.\]
The same comparison therefore gives the remaining two cases and completes the proof.
\end{proof}

\begin{lemma}\label{lem:bound}
Fix $a,b$ with $a<b<a+1$. There exist constants $A_{a,b}>0$ and $B_{a,b}>0$, independent of $d$, such that, for every integer $d\ge6$ satisfying \eqref{case-1},
    \[\frac{\mu_3-(q-1)}{q-2}-\mu_3x\ge A_{a,b}-B_{a,b}x\qquad\text{for all }0<x<1.\]
This estimate holds in all three cases of Lemma \ref{lem:third-eigenvalue}. In particular, it holds for $0<x<A_{a,b}/B_{a,b}$.

One may take $B_{a,b}=A_{a,b}+2$. If $a\ge0$, set
    \[A_{a,b}:=\min\left\{1,\frac{b-a}{1+a-b}\right\}.\]
If $a<0$, let
    \[D:=\left\lceil\max\left\{6,\frac{24|a|}{b-a}\right\}\right\rceil\]
and define the finite set
    \[\mathcal J(a,b):=\left\{k\in\mathbb N:6\le k<D,\quad \frac{q_k}{2}<s_k<q_k-1\right\},\]
where $q_k$ and $s_k$ denote the previously defined $q$ and $s$ evaluated in dimension $k$. Set
\[A_{a,b}:=\min\left(\left\{1,\frac{b-a}{2(1+a-b)}\right\}\cup\left\{\frac{(2s_k-q_k)(s_k+q_k-1)}{q_k(q_k-2)}:\ k\in\mathcal J(a,b)\right\}\right).\]
\end{lemma}
\begin{proof}
Set
    \[\beta:=\frac{\mu_3-(q-1)}{q-2}.\]
We prove $\beta\ge A_{a,b}$ by treating the three cases in Lemma \ref{lem:third-eigenvalue}.

\noindent\textit{First case: $a_c^*\le a<a_c$.}
Here $\mu_3=(q-1)(3q-4)/q$, so
    \[\beta=\frac{2(q-1)}q\ge1\ge A_{a,b}.\]

\noindent\textit{Second case: $a<a_c^*$ and $b\ge b_{\mathrm{FS}}^*(a)$.}
The third eigenvalue has the same expression as in the first case, and hence the same lower bound holds.

\noindent\textit{Third case: $a<a_c^*$ and $b<b_{\mathrm{FS}}^*(a)$.}
In this case, $q/2<s<q-1$ and $\mu_3=\mu_{1,0}$. Therefore,
    \[\beta=\frac{(2s-q)(s+q-1)}{q(q-2)}>0.\]
For comparison, let $s_0:=d/(d-2)$. A direct calculation gives
    \[\frac{(2s_0-q)(s_0+q-1)}{q(q-2)}
    =\frac{b-a}{1+a-b}\frac{d^2-d-2(1+a-b)}{(d-2)^2}
    \ge\frac{b-a}{1+a-b},\]
since $3d-4-2(1+a-b)>0$.
The function $t\mapsto(2t-q)(t+q-1)$ is increasing for $t>1$. If $a\ge0$, then $s\ge s_0$, and thus
    \[\beta\ge\frac{b-a}{1+a-b}\ge A_{a,b}.\]

Suppose now that $a<0$. For $\tau\le0$,
\[\begin{aligned}
0<\frac{\rd}{\rd\tau}\sqrt{1+\frac{4(d-1)}{(d-2-2\tau)^2}}
&=\frac{8(d-1)}{(d-2-2\tau)^2\sqrt{(d-2-2\tau)^2+4(d-1)}}\\
&\le\frac{8(d-1)}{d(d-2)^2}.
\end{aligned}\]
Integrating over $[a,0]$ yields
    \[0<s_0-s\le\frac{8|a|(d-1)}{d(d-2)^2}.\]
Moreover,
\[
 1<s<s_0\le\frac32,\qquad
 \frac{2(s+s_0)+q-2}{q}\le3.
\]
Using
    \[\frac1{q-2}=\frac{d-2(1+a-b)}{4(1+a-b)}\le\frac{d}{4(1+a-b)},\]
we obtain
\[\begin{aligned}
\beta
&=\frac{(2s_0-q)(s_0+q-1)}{q(q-2)}
  -\frac{(s_0-s)[2(s+s_0)+q-2]}{q(q-2)}\\
&\ge\frac{b-a}{1+a-b}-\frac{3(s_0-s)}{q-2}\\
&\ge\frac{b-a}{1+a-b}-\frac{6|a|(d-1)}{(1+a-b)(d-2)^2}\\
&\ge\frac{b-a}{1+a-b}-\frac{12|a|}{(1+a-b)d},
\end{aligned}\]
where we used $d(d-1)/(d-2)^2\le2$ for $d\ge6$.
Therefore, if $d\ge D$, then
    \[\beta\ge\frac{b-a}{2(1+a-b)}\ge A_{a,b}.\]
If $6\le d<D$, then $d\in\mathcal J(a,b)$, and the finite minimum defining $A_{a,b}$ gives the same bound.
Every term in that minimum is strictly positive, and the set of dimensions involved depends only on $a,b$. Thus $A_{a,b}>0$ is independent of $d$.

We have proved $\beta\ge A_{a,b}$ in all three cases. Finally, $d\ge6$ and $0<1+a-b<1$ imply $2<q<3$. Hence $1-(q-2)x>0$ for $0<x<1$, and
\[\begin{aligned}
\frac{\mu_3-(q-1)}{q-2}-\mu_3x
&=\beta\bigl[1-(q-2)x\bigr]-(q-1)x\\
&\ge A_{a,b}\bigl[1-(q-2)x\bigr]-(q-1)x\\
&\ge A_{a,b}-(A_{a,b}+2)x.
\end{aligned}\]
Since $0<A_{a,b}/B_{a,b}<1$, this also proves the assertion on $0<x<A_{a,b}/B_{a,b}$.
\end{proof}

\begin{remark}\label{rem:diagonal-parameters}
The strict inequality $b>a$ cannot in general be replaced by $b\ge a$ if a positive constant $A_{a,b}$ independent of $d$ is required.
Indeed, fix $a=b>0$.
As $d\to\infty$,
\[q=2+\frac4d+\frac8{d^2}+O(d^{-3}),\qquad s=1+\frac2d+\frac{8a+4}{d^2}+O(d^{-3}).\]
Consequently,
\[2s-q=\frac{16a}{d^2}+O(d^{-3})>0,\qquad q-1-s=\frac2d+O(d^{-2})>0\]
for all sufficiently large $d$. Thus these parameters are in the non-degenerate case and lie on the third spectral branch for all sufficiently large dimensions.
However,
\[\frac{\mu_3-(q-1)}{q-2}=\frac{(2s-q)(s+q-1)}{q(q-2)}=\frac{4a}{d}+O(d^{-2})\longrightarrow0.\]
If positive constants $A_{a,a}$ and $B_{a,a}$, independent of $d$, satisfied
\[\frac{\mu_3-(q-1)}{q-2}-\mu_3x\ge A_{a,a}-B_{a,a}x\qquad\text{for all }0<x<A_{a,a}/B_{a,a},\]
then letting $x\downarrow0$ would bound $A_{a,a}$ by the normalized gap above for every sufficiently large $d$. Passing to the limit $d\to\infty$ would imply $A_{a,a}\le0$, a contradiction.
\end{remark}

\subsubsection{Fourth moments of eigenfunctions below a spectral cutoff}
Define
\[
 \eta_i:=\frac{2}{q-2}\sqrt{1+\frac{i(i+d-2)}{\La}},\qquad
 F_{j,\eta}(z):=\leftidx{_2}F_1\sbr{-j,2\eta+j+1;\eta+1;z}.
\]
The hypergeometric equation and derivative identity give
\be\label{10-23-3} z(1-z)F_{j,\eta_i}''(z)+(\eta_i+1)(1-2z)F_{j,\eta_i}'(z)+j(2\eta_i+j+1)F_{j,\eta_i}(z)=0,\ee
and
\be\label{10-23-4} (\eta_i+1)F_{j,\eta_i}'(z)+j(2\eta_i+j+1)F_{j-1,\eta_i+1}(z)=0.\ee
For the rest of this subsection, fix $a,b$ with $a<b<a+1$ and let $d\ge2$ satisfy \eqref{case-1}. We prove the following estimate.
To specify a constant independent of $d$, set
\[d_*:=\max\{2,\lfloor2a+2\rfloor+1\},\qquad
\rho_a:=\max\left\{1,\frac{d_*}{d_*-2-2a}\right\},\]
and, for each integer $K\ge1$, let
\be\label{eq:angular-cutoff-size}
M_K:=\left\lceil\frac{K(K+2)\rho_a^2}{2(1+a-b)}\right\rceil.
\ee
\begin{lemma}\label{lemma:large eigenfunction}
For every integer $K\ge1$, every eigenvalue $\mu\le\mu_{K,0}$, and every $v\in E_\mu$, one has
\be \label{eq:target}
 \left(\int_{\CR}V^{q-4}v^4\rd\nu\right) \left(\int_{\CR}V^q\rd\nu\right) \le D_{K,a,b}\left(\int_{\CR}V^{q-2}v^2\rd\nu\right)^2,
\ee
where one may take
\be\label{eq:DK}
D_{K,a,b}=16^{M_K+1}\left(1-\left(\frac34\right)^{K+1}\right)^2
\frac{(2M_K+1)!}{4^{M_K}(M_K!)^2}.
\ee
For fixed $a,b$, this constant is independent of $d$.
\end{lemma}

We first establish the bounds on the spectral indices needed in the proof.
Since the function $x\mapsto x(x+1)$ is strictly increasing on $(0,\infty)$, formula \eqref{mu-ij} gives
\be \label{eq:radialcut} \mu_{i,j}\le\mu_{K,0} \quad\Longleftrightarrow\quad \eta_i+j\le \eta_K. \ee
Since $\eta_i$ is strictly increasing in $i$, this implies $i\le K$.
Also, $d\ge d_*$ and $d/(d-2-2a)\le\rho_a$. Rationalizing the square root and using $d\ge2$ gives
\[\begin{aligned}
\eta_K-\eta_0
&\le\frac{[d-2(1+a-b)]K(K+d-2)}{(1+a-b)(d-2-2a)^2}\\
&\le\frac{K(K+2)\rho_a^2}{2(1+a-b)}\le M_K.
\end{aligned}\]
Moreover, \eqref{case-1} implies 
    \[\f{4}{q-2}\ge d-2,\quad \f{4}{(q-2)^2\Lambda}>\frac{\f{4}{q-2}+1}{d-1}\ge1,\]
and direct calculation gives
\be\label{eq:eta-lower}\begin{aligned}
 \eta_i^2-\left(\frac{2}{q-2}+i\right)^2
 &=i\left[(i-1)\left(\frac4{(q-2)^2\La}-1\right)\right.\\
 &\hspace{15mm}\left.+\frac{4(d-1)}{(q-2)^2\La}
 -\left(\frac4{q-2}+1\right)\right]>0,
\end{aligned}\ee 
for every $i\ge1$.
Thus $\eta_i>\f{2}{q-2}+i$ for $i\ge1$. Combining this with \eqref{eq:radialcut}, we obtain
\begin{equation}\label{eq:jointcut}
\mu_{i,j}\le\mu_{K,0} \quad\Longrightarrow\quad
 \begin{cases}
   0\le i\le K,\\
   i+j\le M_K.
 \end{cases}
\end{equation}
Consequently, there are at most
\be\label{eq:cutoff-count} N_{K,a,b}:=(K+1)(M_K+1) \ee
distinct eigenvalues not exceeding $\mu_{K,0}$. This bound counts the admissible pairs $(i,j)$, without counting the multiplicity of each spherical harmonic space.

\begin{lemma}\label{lem:radial}
Let $M\ge1$ be an integer, and suppose
\[
 \eta\ge\frac{2}{q-2},\qquad j\in\mathbb Z_{\ge0},\qquad \eta+j\le \frac{2}{q-2}+M.
\]
Define 
\[
 b_M:=\prod_{r=1}^{M}\frac{r+\tfrac12}{r}
      =\frac{\Gamma(M+\tfrac32)}{\Gamma(\tfrac32)\Gamma(M+1)}
      =\frac{(2M+1)!}{4^M(M!)^2}.
\]
Then
\begin{equation}\label{eq:radialratio}
 \mathcal R_{j,\eta}:=
 \frac{\displaystyle
     \int_0^1(z(1-z))^{2\eta-\f{2}{q-2}}F_{j,\eta}^4\rd z
     \int_0^1(z(1-z))^\f{2}{q-2}\rd z}
      {\displaystyle\left(\int_0^1(z(1-z))^{\eta} F_{j,\eta}^2\rd z\right)^2}
 \le 16^j b_M.
\end{equation}
\end{lemma}

\begin{proof}
The parameter range includes the radial degree $i=0$ and the parameters obtained by differentiation. First let $j=0$, so that $F_{0,\eta}\equiv1$, and put
\[
 J(t):=\int_0^1[4z(1-z)]^t\rd z
      =\frac{\sqrt\pi\,\Gamma(t+1)}{2\Gamma(t+\tfrac32)},\qquad t\ge0.
\]
The powers of $4$ cancel in the quotient, so
\[
 \mathcal R_{0,\eta}
 =\frac{J(2\eta-\f{2}{q-2})J(\f{2}{q-2})}{J(\eta)^2}.
\]
Since $0\le4z(1-z)\le1$, the function $J$ is decreasing. Consequently,
\begin{align}
 \mathcal R_{0,\eta}
 &\le\frac{J(\f{2}{q-2})}{J(\eta)}
 \le\frac{J(\f{2}{q-2})}{J(\f{2}{q-2}+M)}\notag\\
 &=\prod_{r=1}^{M}\frac{\f{2}{q-2}+r+\tfrac12}{\f{2}{q-2}+r}
 \le\prod_{r=1}^{M}\frac{r+\tfrac12}{r}=b_M.\label{eq:base}
\end{align}

For $j\ge1$, abbreviate
\[
 F=F_{j,\eta},\qquad
 L=j(2\eta+j+1)>0.
\]
The hypergeometric equation and derivative identity are
\begin{equation}\label{eq:hypid}
 z(1-z)F''+(\eta+1)(1-2z)F'+LF=0,
 \qquad
 F'=-\frac{L}{\eta+1}F_{j-1,\eta+1}.
\end{equation}
Multiply the differential equation by $(z(1-z))^{2\eta-\f{2}{q-2}} F^3$ and integrate by parts. Since $F$ is a polynomial and $2\eta-\f{2}{q-2}>0$, all boundary terms vanish. We get
\begin{equation}\label{eq:IBP4}\begin{aligned}
 L\int_0^1(z(1-z))^{2\eta-\f{2}{q-2}} F^4\rd z
 &=3\int_0^1(z(1-z))^{2\eta-\f{2}{q-2}+1}F^2(F')^2\rd z\\
 &\quad+\left(\eta-\frac{2}{q-2}\right)
 \int_0^1(z(1-z))^{2\eta-\f{2}{q-2}}(1-2z)F^3F'\rd z.
\end{aligned}\end{equation}
A second integration by parts yields
\[\begin{aligned}
 \int_0^1(z(1-z))^{2\eta-\f{2}{q-2}} (1-2z)F^3F'\rd z
 &=-\frac{2\eta-\f{2}{q-2}}4\int_0^1(z(1-z))^{2\eta-\f{2}{q-2}-1}(1-2z)^2F^4\rd z\\
   &~\quad +\frac12\int_0^1(z(1-z))^{2\eta-\f{2}{q-2}} F^4\rd z\notag\\
 &\le\frac12\int_0^1(z(1-z))^{2\eta-\f{2}{q-2}} F^4\rd z.\label{eq:cross}
\end{aligned}\]
Here the integral with exponent $2\eta-\f{2}{q-2}-1$ is finite because $2\eta-\f{2}{q-2}>0$. Since $\eta-\f{2}{q-2}\ge0$, Cauchy--Schwarz applied in \eqref{eq:IBP4} gives
\begin{align*}
 \left(L-\frac{\eta-\f{2}{q-2}}{2}\right)\int_0^1(z(1-z))^{2\eta-\f{2}{q-2}} F^4\rd z
 \le & 3\left(\int_0^1(z(1-z))^{2\eta-\f{2}{q-2}} F^4\rd z\right)^{1/2}\\
       &~\quad \times\left(\int_0^1(z(1-z))^{2\eta-\f{2}{q-2}+2}(F')^4\rd z\right)^{1/2}.
\end{align*}
Thus
\begin{equation}\label{eq:fourth-step}
 \int_0^1(z(1-z))^{2\eta-\f{2}{q-2}} F^4\rd z
 \le\frac9{(L-\tfrac{\eta-\f{2}{q-2}}{2})^2}
       \int_0^1(z(1-z))^{2\eta-\f{2}{q-2}+2}(F')^4\rd z.
\end{equation}
The self-adjoint form
\[
 \bigl((z(1-z))^{\eta+1}F'\bigr)'+L(z(1-z))^\eta F=0
\]
similarly gives
\begin{equation}\label{eq:second-step}
 \int_0^1(z(1-z))^{\eta+1}(F')^2\rd z
   =L\int_0^1(z(1-z))^\eta F^2\rd z.
\end{equation}
Combining \eqref{eq:hypid}, \eqref{eq:fourth-step}, and \eqref{eq:second-step}, with all constant multiples of $F'$ cancelling, we find
\begin{equation}\label{eq:recurrence}
 \mathcal R_{j,\eta}
 \le\frac9{\left(1-\dfrac{\eta-\f{2}{q-2}}{2j(2\eta+j+1)}\right)^2}
             \mathcal R_{j-1,\eta+1}
 \le16\,\mathcal R_{j-1,\eta+1}.
\end{equation}
Indeed,
\[
 0\le\frac{\eta-\f{2}{q-2}}{2j(2\eta+j+1)}<\frac14.
\]
Each differentiation preserves $\eta+j$. Iterating \eqref{eq:recurrence} and using \eqref{eq:base} with $\eta+j\le\f{2}{q-2}+M$ proves
\[
 \mathcal R_{j,\eta}\le16^j\mathcal R_{0,\eta+j}\le16^jb_M.
\]
\end{proof}

\begin{proof}[Proof of Lemma \ref{lemma:large eigenfunction}]
Fix an eigenvalue $\mu\le\mu_{K,0}$.
For each angular degree $i$, strict monotonicity of $\mu_{i,j}$ in $j$ shows that at most one integer $j_i\ge0$ satisfies $\mu=\mu_{i,j_i}$. 
The one-dimensional radial eigenspace of \eqref{equ-eigen-R} then implies
\begin{equation}\label{eq:grouped}\begin{aligned}
 v(t,\omega)&=\sum_{i\in I_\mu}f_i(t)Y_i(\omega),
 \qquad Y_i=\sum_{l=1}^{l_{i,d}}c_{i,l}\Theta_{i,l},\quad\text{and} \\
 f_i(t)&=\left[1-\tanh^2\left(\frac{q-2}{2}\sqrt\La\,t\right)\right]^{\eta_i/2}
 F_{j_i,\eta_i}\left(\frac{1-\tanh\left(\frac{q-2}{2}\sqrt\La\,t\right)}2\right),
\end{aligned}\end{equation}
where $I_\mu\subset\{0,\ldots,K\}$. Each $Y_i$ is a spherical harmonic of degree $i$ and contains the full angular multiplicity for that degree.

Use the substitution
\[
\begin{aligned}
 &z=\frac{1-\tanh\left(\frac{q-2}{2}\sqrt\La\,t\right)}2,\\
 &1-\tanh^2\left(\frac{q-2}{2}\sqrt\La\,t\right)=4z(1-z),\\
 &|\rd t|=\frac{\rd z}{(q-2)\sqrt\La\,z(1-z)}.
\end{aligned}
\]
All dimensional and scaling factors cancel, and \eqref{eq:radialratio} becomes
\begin{equation}\label{eq:radial-t}
 \int_{\R}V^q\rd t\int_{\R}V^{q-4}f_{i}^4\rd t
 \le16^{j_i} b_{M_K}\left(\int_{\R}V^{q-2}f_{i}^2\rd t\right)^2,
 \qquad \mu_{i,j_i}=\mu\le\mu_{K,0}.
\end{equation}
Duoandikoetxea's reverse H\"older inequality \cite[Theorem 1]{Duo=1987} gives
\begin{equation}\label{eq:angular}
 \int_{\S^{d-1}}Y_i^4\rd\omega
 \le\frac{9^i}{|\S^{d-1}|}\left(\int_{\S^{d-1}}Y_i^2\rd\omega\right)^2.
\end{equation}
Define
\[
 E_i:=\left(\int_{\R}V^{q-2}f_i^2\rd t\right)
       \left(\int_{\S^{d-1}}Y_i^2\rd\omega\right).
\]
The radial estimate \eqref{eq:radial-t} and the angular estimate \eqref{eq:angular} give
\begin{equation}\label{eq:block}
 \left(\int_{\CR}V^{q-4}|f_iY_i|^4\rd\nu\right)^{1/4}
 \le\left(\frac{b_{M_K}}{|\S^{d-1}|\int_{\R}V^q\rd t}\right)^{1/4}
               2^{j_i}3^{i/2}E_i^{1/2}.
\end{equation}
Minkowski's inequality in $L^4(\CR,V^{q-4}\rd\nu)$, followed by Cauchy--Schwarz in the finite index set $I_\mu$, yields
\begin{align}
 \left(\int_{\CR}V^{q-4}v^4\rd\nu\right)^{1/4}
 &\le\left(\frac{b_{M_K}}{|\S^{d-1}|\int_{\R}V^q\rd t}\right)^{1/4}
       \sum_{i\in I_\mu}2^{j_i}3^{i/2}E_i^{1/2}\notag\\
 &\le\left(\frac{b_{M_K}}{|\S^{d-1}|\int_{\R}V^q\rd t}\right)^{1/4}
       \left(\sum_{i\in I_\mu}4^{j_i}3^i\right)^{1/2}
       \left(\sum_{i\in I_\mu}E_i\right)^{1/2}.\label{eq:combine}
\end{align}
Different spherical degrees are orthogonal, and $V$ depends only on $t$. Hence
\begin{equation}\label{eq:energy-sum}
 \sum_{i\in I_\mu}E_i=\int_{\CR}V^{q-2}v^2\rd\nu.
\end{equation}
Furthermore, \eqref{eq:jointcut} implies $j_i\le M_K-i$, so
\begin{equation}\label{eq:geom-sum}
 \sum_{i\in I_\mu}4^{j_i}3^i
 \le4^{M_K}\sum_{i=0}^{K}\left(\frac34\right)^i
 =4^{M_K+1}\left(1-\left(\frac34\right)^{K+1}\right).
\end{equation}
Taking the fourth power in \eqref{eq:combine}, applying \eqref{eq:energy-sum}--\eqref{eq:geom-sum}, and using
\[
 \int_\CR V^q\rd\nu=|\S^{d-1}|\int_\R V^q\rd t,
\]
we obtain
\[
 \left(\int_{\CR}V^{q-4}v^4\rd\nu\right)
 \left(\int_{\CR}V^q\rd\nu\right)
 \le b_{M_K}\left[4^{M_K+1}\left(1-\left(\frac34\right)^{K+1}\right)\right]^2
       \left(\int_{\CR}V^{q-2}v^2\rd\nu\right)^2.
\]
This is \eqref{eq:target} with \eqref{eq:DK}.
\end{proof}

\begin{remark}
For the cutoff $\mu_{K,0}$, the dependence of the constant on $a,b$ cannot in general be removed. Indeed, fix $d=6$ and $q=5/2$, so that $b-a=2/5$, and let $a\uparrow a_c=2$. Then $\Lambda\downarrow0$, $\eta_0=4$, and $\eta_K\to\infty$. The eigenfunction with angular degree $K$ and radial degree $0$ has radial fourth-moment ratio
\[\mathcal R_{0,\eta_K}
=\frac{J(2\eta_K-4)J(4)}{J(\eta_K)^2}
\sim\sqrt{\frac2\pi}J(4)\sqrt{\eta_K}\longrightarrow\infty.\]
Its angular factor is fixed and positive. Hence no bound depending only on $K$ is uniform over all admissible $a,b$ for this cutoff.
\end{remark}

\section{Functions close to the manifold of optimizers}\label{sec:local}
We use the normalized measure
\be \rd \nu_1=\f{1}{\nm{V}_q^q}\rd\nu,\ee 
where $V$ is given by \eqref{V}. 
Since $V$ satisfies \eqref{equ-C} and attains $S_{\La,q}$,
\be\label{10-26-1} \nm{V}_q^{q-2}=S_{\La,q}. \ee 
The corresponding energy and Lebesgue norms are
    \[\nm{u}_{\La,\nu_1}:=\sbr{\int_\CR|\pa_t u|^2+|\nabla_\om u|^2+\La u^2\rd\nu_1}^{\f12}=\f{\nm{u}_\La}{\nm{V}_q^{\f{q}{2}}},\]
and
    \[\nm{u}_{p,\nu_1}:=\sbr{\int_\CR|u|^p\rd\nu_1}^{\f1p}=\f{\nm{u}_p}{\nm{V}_q^{\f{q}{p}}}.\]
By \eqref{10-26-1} and \eqref{CKN-C-S}, the CKN inequality on $(\CR,\nu_1)$ becomes
\be\label{CKN-C-1} \nm{u}_{\La,\nu_1}^2\ge\nm{u}_{q,\nu_1}^2, \quad\forall~u\in H^1(\CR,\nu_1), \ee 
with equality if and only if $u\in\MR$.
We also have
\be \ER(u)=\f{\nm{u}_\La^2-S_{\La,q}\nm{u}_q^2}{\dist_\La(u,\MR)^2}=\f{\nm{u}_{\La,\nu_1}^2-\nm{u}_{q,\nu_1}^2}{\inf_{g\in\MR}\nm{u-g}_{\La,\nu_1}^2}. \ee 

For the translation mode and the nonlinear exponent gap, write
\be\label{theta} U:=-V',\qquad \theta:=q-2. \ee
\begin{theorem}\label{thm:local}
Fix $a,b$ with $a<b<a+1$. There are explicit constants $\ve_0=\ve_0(a,b)>0$ and $\ti\delta=\ti\delta(a,b)\in(0,1)$ such that, for every integer $d\ge2$ satisfying \eqref{case-1} and every $-V\le r\in H^1(\CR,\nu_1)$ with
\be\label{con-1} \nm{r}_{q,\nu_1}^2\le\ti\delta, \ee
and
\be\label{con-2}
\abr{r,V}_{\La,\nu_1}=\abr{r,U}_{\La,\nu_1}=0,
\ee
one has
\be\label{eq:local-stability}
\nm{V+r}_{\La,\nu_1}^2-\nm{V+r}_{q,\nu_1}^2
\ge\theta\ve_0\nm{r}_{\La,\nu_1}^2.
\ee
The constants do not depend on $d$. Choices are given in \eqref{eq:local-epsilon-choice} and \eqref{eq:local-delta-choice} below.
\end{theorem}

\subsection{Truncation of the perturbation}
The following proposition bounds
    \[(V+r)^q-V^q-qV^{q-1}r\]
in terms of the three truncations
\be\begin{aligned}
r_1&:=\min\{r,\ga V\},\\
r_2&:=\min\{(r-\ga V)_+,(M-\ga)V\},\\
r_3&:=(r-MV)_+,
\end{aligned}\ee
where $0<\ga\le M$.

\begin{proposition}\label{lemma-1}
For $\ve>0$, $M>0$ and $0<\ga<M/2$, there is a constant $C_{\ga,\ve,M}>0$ such that, for every $q\in[2,3]$ and $r\ge-V$,
\be\label{est-1}\begin{aligned}
    (V+r)^q-V^q-qV^{q-1}r 
    &\le \sbr{\f{q}{2}(q-1)+2\ga\theta}V^{q-2}r_1^2+\sbr{\f{q}{2}(q-1)+C_{\ga,\ve,M}\theta}V^{q-2}r_2^2 \\
        &~\quad +2V^{q-2}\sbr{r_1 r_2+r_1r_3+ r_2r_3}+(1+\ve\theta)r_3^q.
\end{aligned}\ee
\end{proposition}
\begin{proof}
Divide the claimed inequality by $V^q$ and put $t=r/V$. This is the three-part pointwise estimate used in \cite[Section 2]{Dolbeault etc=2025}. For completeness, the following elementary choice makes the constant explicit:
\[Z:=\left(\frac{100(1+M)}{\min\{\ve,1\}}\right)^2,\]
\be\label{eq:explicit-truncation-constant}
\begin{aligned}
C_{\ga,\ve,M}:=\max\Biggl\{&
\frac{(4+3\ga)^2}{4\ga}+M+2\ga,\\
&\frac{(1+M+Z)^3\log(1+M+Z)+(1+\ve)/(2e)+4\ga(M-\ga)}{(M-\ga)^2}\Biggr\}.
\end{aligned}
\ee
For $-1\le t\le\ga$, use
\[(1+t)^q\le1+qt+\frac{q(q-1)}2t^2+(q-2)t_+^3.\]
For $\ga<t\le M$, set $w=t-\ga$. After subtracting the desired quadratic terms, the remainder divided by $q-2$ is at most
\[-\ga^3+(4\ga+3\ga^2)w+(M+2\ga)w^2
\le\left(\frac{(4+3\ga)^2}{4\ga}+M+2\ga\right)w^2.\]
For $t=M+z$ with $z\ge0$, consider the remainder
\[R(q,z):=(1+M+z)^q-1-q(M+z)-\frac{q(q-1)}2M^2-2Mz
-(1+\ve(q-2))z^q.\]
Since $R(2,z)=0$, the derivative bound
\[\begin{aligned}
\partial_qR(q,z)
&\le(1+M+Z)^3\log(1+M+Z)+\frac{1+\ve}{2e}\\
&\le C_{\ga,\ve,M}(M-\ga)^2-4\ga(M-\ga)
\end{aligned}\]
gives the required estimate when $0\le z\le Z$. For $z\ge Z$, the same cubic Taylor bound yields
\[\frac{R(q,z)}{(q-2)z^q}
\le\frac{(1+M)(1+2\log z)}z
+\frac{(1+M)^2(2+\log z)}{z^2}
+\frac{(1+M)^3}{z^3}-\ve\le0.\]
For the last inequality, the three positive terms decrease for $z\ge Z$; use
\[\log Z\le\sqrt Z\]
and the definition of $Z$. In both ranges, the remaining correction is bounded by
\[(q-2)\bigl((q+1)\ga(M-\ga)-2\ga^3\bigr)
\le4(q-2)\ga(M-\ga),\]
which is included in the choice of $C_{\ga,\ve,M}$.
This proves the claim for $2<q\le3$. At $q=2$ it is an equality.
\end{proof}

For the truncation argument below, assume $d\ge6$, so that $2<q<3$. The remaining dimensions are treated at the end of the section. Assume that $r\ge-V$ satisfies \eqref{con-1}--\eqref{con-2}. We prove the local estimate for $0<\ve_0<1/3$ by choosing $\ti\delta$ sufficiently small.
For auxiliary parameters $\ve_1,\ve_2>0$, apply Proposition \ref{lemma-1} with
\be\label{ve-2}
\ga=\frac{\ve_1}{2},\qquad M=1+2\ga,\qquad \ve=\ve_2,
\qquad C_{\ve_1,\ve_2}:=C_{\ga,\ve_2,M}.
\ee
Then $r=r_1+r_2+r_3$. Since $\nm{V}_{\La,\nu_1}=\nm{V}_{q,\nu_1}=1$ and $\abr{r,V}_{\La,\nu_1}=0$, we get
\[\nm{V+r}_{\La,\nu_1}^2=1+\nm{r}_{\La,\nu_1}^2.\]
Integrating \eqref{est-1}, the linear term vanishes. Using $(1+x)^{2/q}\le 1+\f2qx$ for $q>2$ and $x>0$, we get 
\[\begin{aligned}
\nm{V+r}_{q,\nu_1}^2
&\le1+\left(q-1+\frac{2\ve_1}{q}\theta\right)\int_\CR V^{q-2}r_1^2\rd\nu_1 +\left(q-1+\frac{2C_{\ve_1,\ve_2}}q\theta\right)\int_\CR V^{q-2}r_2^2\rd\nu_1\\
&\quad+\frac2q(1+\ve_2\theta)\int_\CR r_3^q\rd\nu_1
+\frac4q\int_\CR V^{q-2}(r_1r_2+r_1r_3+r_2r_3)\rd\nu_1\\
&\le1+(q-1+\ve_1\theta)\int_\CR V^{q-2}r_1^2\rd\nu_1
+(q-1+C_{\ve_1,\ve_2}\theta)\int_\CR V^{q-2}r_2^2\rd\nu_1\\
&\quad+\frac2q(1+\ve_2\theta)\int_\CR r_3^q\rd\nu_1
+\frac4q\int_\CR V^{q-2}(r_1r_2+r_1r_3+r_2r_3)\rd\nu_1,
\end{aligned}\]
where $2/q\le1$ was used in the last step.
For $0<\ve_0\le\theta^{-1}$, define
\[\begin{aligned}
I_1&:=(1-\theta\ve_0)\nm{r_1}_{\La,\nu_1}^2
-(q-1+\ve_1\theta)\int_\CR V^{q-2}r_1^2\rd\nu_1+\sigma_0\theta\int_\CR V^{q-2}(r_2^2+r_3^2)\rd\nu_1,\\
I_2&:=(1-\theta\ve_0)\nm{r_2}_{\La,\nu_1}^2-(q-1+C_{\ve_1,\ve_2}\theta+\sigma_0\theta)\int_\CR V^{q-2}r_2^2\rd\nu_1,\\
I_3&:=(1-\theta\ve_0)\nm{r_3}_{\La,\nu_1}^2-\frac2q(1+\ve_2\theta)\int_\CR r_3^q\rd\nu_1-\sigma_0\theta\int_\CR V^{q-2}r_3^2\rd\nu_1,\\
I_4&:=2(1-\theta\ve_0)\abr{r_1,r_2}_{\La,\nu_1}-\frac4q\int_\CR V^{q-2}r_1r_2\rd\nu_1,\\
I_5&:=2(1-\theta\ve_0)\abr{r_1+r_2,r_3}_{\La,\nu_1}-\frac4q\int_\CR V^{q-2}(r_1+r_2)r_3\rd\nu_1.
\end{aligned}\]
The preceding estimate becomes
\begin{equation}\label{eq:deficit-decomposition}
\nm{V+r}_{\La,\nu_1}^2-\nm{V+r}_{q,\nu_1}^2
\ge\theta\ve_0\nm{r}_{\La,\nu_1}^2+I_1+I_2+I_3+I_4+I_5.
\end{equation}
We show that $I_4,I_5,I_1,I_3,I_2$ are nonnegative, in this order.

\vskip0.1in
\noindent {\bf Bounds on $I_4$ and $I_5$.}
Since $r_2=0$ on $\{r<\ga V\}$ and $r_1=\ga V$ on $\{r\ge\ga V\}$,
\be\begin{aligned}
    I_4&=2\ga(1-\theta\ve_0)\abr{V,r_2}_{\La,\nu_1}
    -\frac{4\ga}{q}\int_\CR V^{q-1}r_2\rd\nu_1\\
    &=(\tf1q-\ve_0)\ve_1\theta\int_\CR V^{q-1}r_2\rd\nu_1\ge 0,
\end{aligned}\ee 
provided $\ve_0<1/q$. Similarly, we have
\[r_3=0\quad\text{on }\{r<MV\},\qquad
r_1+r_2=MV\quad\text{on }\{r\ge MV\},\]
so that
\be I_5=2M(\tf1q-\ve_0)\theta\int_\CR V^{q-1}r_3\rd\nu_1\ge 0. \ee 

\vskip0.1in
\noindent {\bf Bound on $I_1$.}
Recall that $V\in E_{\mu_1}$ and $U\in E_{\mu_2}$, where $\mu_1=1$, $\mu_2=q-1$ and $\mu_3>q-1$. Let $A_{a,b},B_{a,b}$ be given by Lemma \ref{lem:bound}.
We have
\begin{lemma}\label{lem:I1}
Let $0<\ve_0<A_{a,b}/B_{a,b}$, and set
\be\label{ve-1}
\ve_1:=\frac12(A_{a,b}-B_{a,b}\ve_0),\qquad
C_0:=1+\ve_0+2\ve_1+2\sqrt3(\ve_0+\ve_1).
\ee
For $d\ge6$ and $r$ satisfying \eqref{con-1}--\eqref{con-2}, one has $I_1\ge0$ provided
\be\label{eq:I1-sigma-threshold}
\sigma_0\ge\frac{C_0^2\ti\delta}{2\ve_1\ga^2}.
\ee
\end{lemma}
\begin{proof}
Write
\[\ti r_1=r_1-\beta_1V-\beta_2U,\]
with
\[\beta_1=\int_\CR V^{q-1}r_1\rd\nu_1,\quad \beta_2=\frac{\int_\CR V^{q-2}Ur_1\rd\nu_1}{\int_\CR V^{q-2}U^2\rd\nu_1}.\]
This is the projection onto the complement of $E_{\mu_1}\oplus E_{\mu_2}$, both for the energy inner product and for the weighted $L^2$ inner product. 
In this proof, we put
\[
Q_1:=\int_\CR V^{q-2}r_1^2\rd\nu_1,\quad Q_{23}:=\int_\CR V^{q-2}(r_2^2+r_3^2)\rd\nu_1,
\]
and
\[c_1:=\beta_1^2,\quad  c_2:=\beta_2^2\int_\CR V^{q-2}U^2\rd\nu_1.\]
The spectral gap and Lemma \ref{lem:bound} give
\[\begin{aligned}
\frac{I_1}{\theta}
&\ge\left(\frac{\mu_3-(q-1)}{q-2}-\mu_3\ve_0-\ve_1\right)(Q_1-c_1-c_2)
+\sigma_0Q_{23}\\
&\qquad -(1+\ve_0+\ve_1)c_1-\bigl((q-1)\ve_0+\ve_1\bigr)c_2\\
&\ge\ve_1Q_1+\sigma_0Q_{23}
-(1+\ve_0+2\ve_1)c_1-2(\ve_0+\ve_1)c_2,
\end{aligned}\]
where $q-1\le2$ and $A_{a,b}-B_{a,b}\ve_0=2\ve_1$ were used in the last line.

To control the lost orthogonality, set
\[Y:=\frac{U}{\left(\int_\CR V^{q-2}U^2\rd\nu_1\right)^{1/2}}.\]
For $v=V$ or $v=Y$, orthogonality of $r$ and H\"older's inequality yield
\[\begin{aligned}
\left(\int_\CR V^{q-2}vr_1\rd\nu_1\right)^2
&=\left(\int_\CR V^{q-2}v(r_2+r_3)\rd\nu_1\right)^2\\
&\le\left(\int_\CR V^{q-4}v^4\rd\nu_1\right)^{1/2}
\left(\int_{\{r_2+r_3>0\}}V^q\rd\nu_1\right)^{1/2}
\int_\CR V^{q-2}(r_2+r_3)^2\rd\nu_1.
\end{aligned}\]
Since $\{r_2+r_3>0\}\subset\{r_1\ge\ga V\}$, we have
\[\int_{\{r_2+r_3>0\}}V^q\rd\nu_1\le\frac{Q_1}{\ga^2}.\]
Also, $0\le r_2+r_3\le r_+$, and weighted H\"older gives
\[\int_\CR V^{q-2}(r_2+r_3)^2\rd\nu_1
\le\min\{\ti\delta,2Q_{23}\}\le\sqrt{2\ti\delta Q_{23}}.\]
The normalized fourth moment for $v=V$ is $\int_\CR V^q\rd\nu_1=1$. For $v=Y$, substituting $z=\tanh((q-2)\sqrt\La\,t/2)$ and evaluating the beta integrals gives
\[\int_\CR V^{q-4}Y^4\rd\nu_1
=\frac{\mathrm B(\tfrac12,\tfrac q{q-2})\mathrm B(\tfrac52,\tfrac q{q-2})}{\mathrm B(\tfrac32,\tfrac q{q-2})^2}
=\frac{3(3q-2)}{5q-6}\le3.\]
Consequently,
\[c_1\le\frac{\sqrt{2\ti\delta}}{\ga}\sqrt{Q_1Q_{23}},\qquad
c_2\le\sqrt3\frac{\sqrt{2\ti\delta}}{\ga}\sqrt{Q_1Q_{23}}.\]
It follows that
\[\frac{I_1}{\theta}\ge\ve_1Q_1+\sigma_0Q_{23}
-\frac{C_0\sqrt{2\ti\delta}}{\ga}\sqrt{Q_1Q_{23}}\ge0\]
by \eqref{eq:I1-sigma-threshold} and the nonpositive-discriminant condition for this quadratic form.
\end{proof}

Define
\be\label{delta-1}
\delta_1:=\frac{4\ve_1\ve_2\ga^2}{qC_0^2}.
\ee
\begin{corollary}\label{cor:I1-smallness}
Under the hypotheses of Lemma \ref{lem:I1}, the condition
\be\label{con:ti delta}
0<\ti\delta\le\delta_1\frac{q\sigma_0}{2\ve_2}
\ee
implies $I_1\ge0$.
\end{corollary}

\vskip0.1in
\noindent {\bf Bound on $I_3$.} 
We estimate this term using the CKN inequality.
Let $\ti\delta\in(0,1)$ and $0<\ve_0<\min\{\tf{A_{a,b}}{B_{a,b}},\f13\}$, and take 
\be\label{sigma-0} \ve_2=\frac14(1-3\ve_0),\qquad \sigma_0=\frac{2\ve_2}{q}. \ee
Since $\nm{r_3}_{q,\nu_1}\le\nm{r}_{q,\nu_1}\le 1$ by \eqref{con-1}, and
    \[\int_\CR V^{q-2}r_3^2\rd\nu_1\le\nm{V}_{q,\nu_1}^{q-2}\nm{r_3}_{q,\nu_1}^{2}=\nm{r_3}_{q,\nu_1}^{2},\]
\eqref{CKN-C-1} gives
\[\begin{aligned}
I_3&=(1-\theta\ve_0)\nm{r_3}_{\La,\nu_1}^2-\f2q(1+\ve_2\theta)\int_\CR r_3^q\rd\nu_1-\sigma_0\theta\int_\CR V^{q-2}r_3^2\rd\nu_1\\
&\ge (1-\theta\ve_0-\tf2q(1+\ve_2\theta)-\sigma_0\theta)\nm{r_3}_{q,\nu_1}^2\\
&=\f{\theta}{q}(1-q\ve_0-4\ve_2)\nm{r_3}_{q,\nu_1}^2\\
&\ge \f{\theta}{q}(1-3\ve_0-4\ve_2)\nm{r_3}_{q,\nu_1}^2\ge0,
\end{aligned}\]
where we used $q\le3$ for $d\ge6$.

\vskip0.1in
\noindent {\bf Bound on $I_2$.}
We use the cutoff $\mu_{K,0}$ and retain the energy of the low spectral modes.
\begin{lemma}\label{lem:I2}
For any $0<\ve_0<\min\{\tf{A_{a,b}}{B_{a,b}},\f13\}$, take $\ve_1,\ga$ as in \eqref{ve-1}, \eqref{ve-2}, and $\ve_2,\sigma_0$ as in \eqref{sigma-0}. There is a constant $\delta_2\in(0,1)$, independent of $d$, such that for all $d\ge6$, all $0<\ti\delta\le\delta_2$ and all $r$ as in Theorem \ref{thm:local}, one has $I_2\ge0$.
\end{lemma}
\begin{proof}
Choose an integer
\be\label{eq:I2-choice-K}
K\ge3\left(1+\ve_0+C_{\ve_1,\ve_2}+\ve_2\right).
\ee
This choice is independent of $d$. For $v\in E_\mu$ with $\mu\le\mu_{K,0}$, we claim that
\be\label{10-27-1}
\left|\int_\CR V^{q-2}vr_2\rd\nu_1\right|
\le\ga^{-\f q4}\ti\delta^{\f q8}D_{K,a,b}^{\f14}
\left(\int_\CR V^{q-2}v^2\rd\nu_1\right)^{\f12}
\left(\int_\CR V^{q-2}r_2^2\rd\nu_1\right)^{\f12}.
\ee
Indeed, H\"older's inequality and Lemma \ref{lemma:large eigenfunction}, with $\int_\CR V^q\rd\nu_1=1$, give
\[\begin{aligned}
\left|\int_\CR V^{q-2}vr_2\rd\nu_1\right|
&\le\left(\int_\CR V^{q-4}v^4\rd\nu_1\right)^{\f14}
\left(\int_\CR V^q1_{\{r_2>0\}}\rd\nu_1\right)^{\f14}
\left(\int_\CR V^{q-2}r_2^2\rd\nu_1\right)^{\f12}\\
&\le D_{K,a,b}^{\f14}
\left(\int_\CR V^{q-2}v^2\rd\nu_1\right)^{\f12}
\left(\int_\CR V^q1_{\{r_2>0\}}\rd\nu_1\right)^{\f14}
\left(\int_\CR V^{q-2}r_2^2\rd\nu_1\right)^{\f12}.
\end{aligned}\]
Since $\{r_2>0\}\subset\{r>\ga V\}$,
\[\int_\CR V^q1_{\{r_2>0\}}\rd\nu_1
\le\ga^{-q}\int_\CR |r|^q\rd\nu_1
\le\ga^{-q}\ti\delta^{q/2},\]
which proves \eqref{10-27-1}.

Let $\Pi_\mu$ be the orthogonal projection onto $E_\mu$ in $L^2(\CR,V^{q-2}\rd\nu_1)$, and set
\[\Pi_Kr_2:=\sum_{\mu\le\mu_{K,0}}\Pi_\mu r_2,
\qquad Q_2:=\int_\CR V^{q-2}r_2^2\rd\nu_1.\]
Applying \eqref{10-27-1} to $v=\Pi_\mu r_2$, using orthogonality and the count \eqref{eq:cutoff-count}, yields
\[\begin{aligned}
\int_\CR V^{q-2}(\Pi_Kr_2)^2\rd\nu_1
&=\sum_{\mu\le\mu_{K,0}}\int_\CR V^{q-2}(\Pi_\mu r_2)^2\rd\nu_1\\
&\le N_{K,a,b}\ga^{-q/2}\ti\delta^{q/4}D_{K,a,b}^{1/2}Q_2.
\end{aligned}\]
Take
\be\label{delta-2}
\delta_2:=\min\left\{\frac12,\frac{\ga^2}{4N_{K,a,b}^2D_{K,a,b}}\right\}.
\ee
Since $q\ge2$ and $\ti\delta/\ga^2\le1/(4N_{K,a,b}^2D_{K,a,b})\le1$, we obtain
\[N_{K,a,b}D_{K,a,b}^{1/2}
\left(\frac{\ti\delta}{\ga^2}\right)^{q/4}
\le N_{K,a,b}D_{K,a,b}^{1/2}
\left(\frac{\ti\delta}{\ga^2}\right)^{1/2}\le\frac12.\]
Thus $\int_\CR V^{q-2}(\Pi_Kr_2)^2\rd\nu_1\le Q_2/2$.

All eigenvalues are at least $1$, and the projection is also orthogonal for the energy inner product. Retaining both parts of the energy therefore gives
\[\begin{aligned}
\nm{r_2}_{\La,\nu_1}^2
&\ge\int_\CR V^{q-2}(\Pi_Kr_2)^2\rd\nu_1
+\mu_{K,0}\int_\CR V^{q-2}(r_2-\Pi_Kr_2)^2\rd\nu_1\\
&=Q_2+(\mu_{K,0}-1)
\left(Q_2-\int_\CR V^{q-2}(\Pi_Kr_2)^2\rd\nu_1\right)\\
&\ge\left(1+\frac{\mu_{K,0}-1}{2}\right)Q_2.
\end{aligned}\]
By \eqref{eq:eta-lower}, $\eta_K-\eta_0>K\ge1$. Formula \eqref{mu-ij} consequently gives
\[\begin{aligned}
\frac{\mu_{K,0}-1}{q-2}
&=\frac{(\eta_K-\eta_0)\left[4+(q-2)(\eta_K-\eta_0+1)\right]}{2q}\ge\eta_K-\eta_0\ge K.
\end{aligned}\]
Using $0<\theta=q-2<1$, $\ve_0<1/3$ and $\sigma_0=2\ve_2/q\le\ve_2$, we conclude that
\[\begin{aligned}
I_2
&\ge\theta\left[\frac{1-\theta\ve_0}{2}
\frac{\mu_{K,0}-1}{q-2}
-(1+\ve_0+C_{\ve_1,\ve_2}+\sigma_0)\right]Q_2\\
&\ge\theta\left[\frac K3
-(1+\ve_0+C_{\ve_1,\ve_2}+\ve_2)\right]Q_2\ge0
\end{aligned}\]
by \eqref{eq:I2-choice-K}. The constants $K$, $M_K$, $D_{K,a,b}$, $N_{K,a,b}$ and $\delta_2$ depend only on the fixed parameters and are independent of $d$.
\end{proof}

\subsection{Completion of the proof}
\begin{proof}[Proof of Theorem \ref{thm:local}]
We follow the high-dimensional truncation and finite-dimensional perturbation strategy of \cite[Section 2]{Dolbeault etc=2025}.

\noindent\textit{A fixed-dimensional estimate.}
For each fixed admissible dimension, the orthogonality conditions and the eigenvalue equations give
\[\int_\CR V^{q-1}r\rd\nu_1=0,\qquad
\int_\CR V^{q-2}r^2\rd\nu_1\le\frac{\nm{r}_{\La,\nu_1}^2}{\mu_3}.\]
Define
\[L(q):=\begin{cases}
\displaystyle\frac2q,&2<q\le3,\\[1mm]
\displaystyle\frac23(q-1)(q-2)\max\{1,2^{q-4}\},&q>3.
\end{cases}\]
We claim that, whenever \eqref{con-1} holds with $\ti\delta\le1$,
\be\label{eq:fixed-dimension-deficit}
\nm{V+r}_{\La,\nu_1}^2-\nm{V+r}_{q,\nu_1}^2
\ge\left(1-\frac{q-1}{\mu_3}
-L(q)\ti\delta^{\min\{q-2,1\}/2}\right)\nm{r}_{\La,\nu_1}^2.
\ee
For $2<q\le3$ and $t\ge-1$, use
\[(1+t)^q\le1+qt+\frac{q(q-1)}2t^2+t_+^q.\]
For $q>3$, Taylor's integral formula gives
\[(1+t)^q\le1+qt+\frac{q(q-1)}2t^2
+\frac{q(q-1)(q-2)}6\max\{1,2^{q-4}\}\bigl(|t|^3+|t|^q\bigr).\]
Indeed, for $-1\le t\le0$ the third-order remainder is nonpositive, while for $t\ge0$ it is bounded by
\[\frac{q(q-1)(q-2)}6t^3(1+t)^{q-3},\qquad
(1+t)^{q-3}\le\max\{1,2^{q-4}\}(1+t^{q-3}).\]
Apply these estimates with $t=r/V$, multiply by $V^q$, and integrate. 
For $2\le p\le q$, weighted H\"older and \eqref{CKN-C-1} give
\[\int_\CR V^{q-p}|r|^p\rd\nu_1
\le\nm{r}_{q,\nu_1}^p
\le\ti\delta^{(p-2)/2}\nm{r}_{\La,\nu_1}^2,\]
where $\int_\CR V^q\rd\nu_1=1$ is used. Concavity of $x^{2/q}$ then proves \eqref{eq:fixed-dimension-deficit}. For $q>3$ and $\ti\delta\le1$, both remainder terms are bounded by a constant times
\[\ti\delta^{1/2}\nm{r}_{\La,\nu_1}^2.\]

\noindent\textit{Constants for the finitely many small dimensions.}
Let $\mathcal D_{\mathrm{low}}(a,b)$ be the set of integers $m\in\{2,3,4,5\}$ for which the fixed pair $a,b$ satisfies \eqref{case-1}. Write $q_m$ and $\mu_3(m)$ for the corresponding parameters in dimension $m$, and set
\[\kappa_m:=1-\frac{q_m-1}{\mu_3(m)}.\]
Define
\[\begin{aligned}
\ve_{\mathrm{low}}
&:=\min\left(\{1\}\cup\left\{\frac{\kappa_m}{2(q_m-2)}:m\in\mathcal D_{\mathrm{low}}(a,b)\right\}\right),\\
\delta_{\mathrm{low}}
&:=\min\left(\left\{\frac12\right\}\cup
\left\{\left(\frac{\kappa_m}{2L(q_m)}\right)^{2/\min\{q_m-2,1\}}:m\in\mathcal D_{\mathrm{low}}(a,b)\right\}\right).
\end{aligned}\]
Both constants are positive and depend only on $a,b$; these definitions also cover the case when the set of small admissible dimensions is empty.

\noindent\textit{A common choice of parameters.}
Choose
\be\label{eq:local-epsilon-choice}
\ve_0:=\frac12\min\left\{\frac13,\frac{A_{a,b}}{B_{a,b}},\ve_{\mathrm{low}}\right\}.
\ee
Choose $\ve_1,C_0$ according to \eqref{ve-1}, choose $\ve_2$ according to \eqref{sigma-0}, and take $\ga,M,C_{\ve_1,\ve_2}$ from \eqref{ve-2}. Thus the constant in the truncation estimate is given explicitly by Proposition \ref{lemma-1}. Set
\[K:=\left\lceil3\bigl(1+\ve_0+C_{\ve_1,\ve_2}+\ve_2\bigr)\right\rceil,\]
and use $D_{K,a,b}$ and $N_{K,a,b}$ from \eqref{eq:DK} and \eqref{eq:cutoff-count}. Finally, define
\be\label{eq:local-delta-choice}
\ti\delta:=\min\left\{\delta_{\mathrm{low}},
\frac{4\ve_1\ve_2\ga^2}{3C_0^2},
\frac{\ga^2}{4N_{K,a,b}^2D_{K,a,b}}\right\}.
\ee
All these constants are positive and depend only on the fixed pair $a,b$.

For $d\ge6$, we have $2<q<3$ and take $\sigma_0$ from \eqref{sigma-0}. In particular,
\[\sigma_0\ge\frac{2\ve_2}{3}
\ge\frac{C_0^2\ti\delta}{2\ve_1\ga^2},\]
so Lemma \ref{lem:I1} gives $I_1\ge0$. The choice of $K$ and the last term in \eqref{eq:local-delta-choice} give $I_2\ge0$ by Lemma \ref{lem:I2}. The estimate for $I_3$ applies because $\ti\delta\le1$ and $\ve_0<1/3$. Also $\ve_0<1/3\le1/q$, so $I_4,I_5\ge0$. Substitution into \eqref{eq:deficit-decomposition} proves the asserted inequality in every admissible dimension $d\ge6$.

If $d=m\in\mathcal D_{\mathrm{low}}(a,b)$, then \eqref{eq:fixed-dimension-deficit} and $\ti\delta\le\delta_{\mathrm{low}}$ imply
\[\nm{V+r}_{\La,\nu_1}^2-\nm{V+r}_{q,\nu_1}^2
\ge\frac{\kappa_m}{2}\nm{r}_{\La,\nu_1}^2
\ge(q_m-2)\ve_0\nm{r}_{\La,\nu_1}^2\]
by \eqref{eq:local-epsilon-choice}. This treats all remaining admissible dimensions and completes the proof.
\end{proof}

\section{From local to global stability}\label{sec:global}
Throughout this section we assume \eqref{ourcase}. In particular,
\[
2<q<2^*,\qquad 0<\La<\La_{FS},\qquad q':=\frac q{q-1}.
\]
All norms and inner products use the original cylinder measure $\nu$. Write
\[\begin{aligned}
E_\La(u)&:=\nm{u}_\La^2,\quad D_{\La,q}(u):=E_\La(u)-S_{\La,q}\nm{u}_q^2,\quad \rho(u):=\frac{\dist_\La(u,\MR)^2}{E_\La(u)}\quad(u\ne0).
\end{aligned}\]
Since $0\in\MR$, we have $0\le\rho(u)\le1$. Define
\[\begin{aligned}
c_{\La,q}^{+}&:=\inf_{\substack{u\ge0\\u\notin\MR}}
\frac{D_{\La,q}(u)}{\dist_\La(u,\MR)^2},\\
I_{\La,q}(\delta)&:=\inf_{\substack{u\ge0,\ u\notin\MR\\\rho(u)\le\delta}}
\frac{D_{\La,q}(u)}{\dist_\La(u,\MR)^2},\qquad0<\delta<1.
\end{aligned}\]
We follow the local-to-global strategy of \cite{Dolbeault etc=2025}. The path on the cylinder combines axial continuous Steiner symmetrization with normalized inverse iteration. We give the construction below to track its effect on the constants.

\subsection{The relative neighborhood of the optimizer manifold}
\begin{lemma}\label{lem:nearest-relative}
Every nonzero nonnegative $u\in H^1(\CR)$ has a nearest optimizer
\[
g=cV(\cdot-\tau),\qquad c>0.
\]
If $v=u-g$, then
\be\label{eq:nearest-orthogonality}
\abr{v,g}_\La=\abr{v,\partial_tg}_\La=0,\qquad
E_\La(u)=E_\La(g)+E_\La(v).
\ee
Let $\ve_0,\ti\delta$ be the constants of Theorem \ref{thm:local}, and set
\be\label{eq:relative-radius}
\delta_0:=\frac{\ti\delta}{1+\ti\delta}.
\ee
Then
\be\label{eq:local-relative-coefficient}
I_{\La,q}(\delta_0)\ge(q-2)\ve_0.
\ee
\end{lemma}
\begin{proof}
Write $V_\tau=V(\cdot-\tau)$. The equation for $V_\tau$ gives
\[\abr{u,V_\tau}_\La=\int_\CR uV_\tau^{q-1}\rd\nu>0.\]
This correlation is continuous in $\tau$ and tends to zero as $|\tau|\to\infty$. This follows by approximating the two factors in $L^q$ and $L^{q'}$ by functions with compact axial support. For fixed $\tau$, minimization over $c\in\R$ gives
\[c=\frac{\abr{u,V_\tau}_\La}{E_\La(V)},\qquad
\inf_cE_\La(u-cV_\tau)=E_\La(u)-\frac{\abr{u,V_\tau}_\La^2}{E_\La(V)}.\]
The positive maximum of the correlation is attained, so a nearest optimizer exists and has positive amplitude. Differentiation in the amplitude and translation parameters proves \eqref{eq:nearest-orthogonality}.

Suppose $\rho(u)\le\delta_0$. Translate and rescale by this nearest optimizer, and write
\[r(t,\omega)=c^{-1}u(t+\tau,\omega)-V(t).\]
Then $r\ge-V$ and the orthogonality conditions \eqref{con-2} hold. Since $E_\La(V)=\nm{V}_q^q$, the normalized energy of $r$ is
\[\nm{r}_{\La,\nu_1}^2
=\frac{E_\La(v)}{E_\La(g)}
=\frac{\rho(u)}{1-\rho(u)}\le\ti\delta.\]
By \eqref{CKN-C-1}, $\nm{r}_{q,\nu_1}^2\le\ti\delta$ as well. Theorem \ref{thm:local} therefore applies. Returning to $u$, both its deficit and its squared distance acquire the same factor:
\[\begin{aligned}
D_{\La,q}(u)
&=c^2\nm{V}_q^q\bigl(\nm{V+r}_{\La,\nu_1}^2-\nm{V+r}_{q,\nu_1}^2\bigr),\\
\dist_\La(u,\MR)^2&=c^2\nm{V}_q^q\nm{r}_{\La,\nu_1}^2.
\end{aligned}\]
Thus
\[
D_{\La,q}(u)\ge(q-2)\ve_0\dist_\La(u,\MR)^2,
\]
which proves \eqref{eq:local-relative-coefficient}.
\end{proof}

\subsection{An energy-controlled path}
The next lemma combines continuous rearrangement with normalized inverse iteration.

\begin{lemma}\label{lem:energy-path}
Assume $2<q<2^*$ and $0<\La<\La_{FS}$. For every nonnegative $u\in H^1(\CR)$ with $\nm{u}_q=1$, there is a continuous map $\Gamma:[0,1]\to H^1(\CR)$ such that
\[\Gamma(0)=u,\qquad \Gamma(1)=\frac{V}{\nm{V}_q},\qquad
\Gamma(\tau)\ge0,\qquad \nm{\Gamma(\tau)}_q=1,\qquad
E_\La(\Gamma(\tau))\le E_\La(u).\]
\end{lemma}
\begin{proof}
\noindent\textit{Inverse steps and their continuous connections.}
Let
\[
\mathcal A:=-\partial_t^2-\Delta_\omega+\La.
\]
The sharp inequality and the Lax--Milgram theorem give
\[\nm{\mathcal A^{-1}f}_\La\le S_{\La,q}^{-1/2}\nm{f}_{q'},\]
and testing with the negative part shows that $\mathcal A^{-1}$ preserves nonnegativity. For $u\ge0$, $\nm{u}_q=1$, put
\be\label{eq:inverse-map}
z:=\mathcal A^{-1}(u^{q-1}),\qquad e:=E_\La(u),\qquad
T(u):=\frac{z}{\nm{z}_q}.
\ee
Then
\[
\abr{u,z}_\La=1,\qquad 1\le eE_\La(z),\qquad
E_\La(z)=\int_\CR u^{q-1}z\rd\nu\le\nm{z}_q.
\]
Consequently,
\be\label{10-29-1} E_\La(Tu)\le\frac1{E_\La(z)}\le E_\La(u).\ee
Moreover, $T:L^q_+\setminus\{0\}\to H^1(\CR)$ is continuous. 
This follows from continuity of $u\mapsto u^{q-1}$ from the nonnegative cone $L^q_+$ to $L^{q'}$, the inverse bound, and the embedding $H^1\hookrightarrow L^q$.
Set
\[
B:=eE_\La(z)-1\ge0,
\]
and define
\[v_\tau:=(1-\tau)u+\tau ez,\qquad \gamma_u(\tau):=\frac{v_\tau}{\nm{v_\tau}_q},\qquad 0\le \tau\le1.\]
The path $\gamma_u$ joins $u$ to $Tu$. Direct calculation and H\"older's inequality give
\[E_\La(v_\tau)=e(1+\tau^2B),\qquad
\nm{v_\tau}_q\ge\int_\CR u^{q-1}v_\tau\rd\nu=1+\tau B.\]
Thus its denominator never vanishes and
\[E_\La(\gamma_u(\tau))\le e\frac{1+\tau^2B}{(1+\tau B)^2}\le e.\]

\noindent\textit{Entering a compact symmetry cone.}
Let $R_\tau$ be continuous Steiner symmetrization in $t$, applied on each fiber $\mathbb R\times\{\omega\}$, with $R_0u=u$ and $R_\infty u$ even and decreasing in $|t|$. The one-dimensional rearrangement properties give preservation of $L^p$ norms, $L^p$ contraction, and $L^p$ continuity on $[0,\infty]$ for finite $p$; see \cite{Dolbeault etc=2025}. Fubini gives the axial derivative estimate. For a one-parameter angular rotation $J_h$, commutation with $R_\tau$ and $L^2$ contraction imply
\[\left\|\frac{J_hR_\tau u-R_\tau u}{h}\right\|_2
\le\left\|\frac{J_hu-u}{h}\right\|_2.\]
Passing to the standard rotational Killing fields $X_{ij}$ and using
\[
\sum_{i<j}|X_{ij}w|^2=|\nabla_\omega w|^2
\]
yields
\[\nm{R_\tau u}_2=\nm{u}_2,\qquad
\nm{\partial_tR_\tau u}_2\le\nm{\partial_tu}_2,\qquad
\nm{\nabla_\omega R_\tau u}_2\le\nm{\nabla_\omega u}_2.\]
In particular,
\[
E_\La(R_\tau u)\le E_\La(u).
\]
The $L^q$ continuity follows by dominated convergence on the fibers. Therefore $\tau\mapsto T(R_\tau u)$ is continuous in $H^1$. Preceding this path with $\gamma_u$ joins $u$ to
\[
u_0=T(R_\infty u),
\]
with unit $L^q$ norm and energy bounded by $E_\La(u)$.

Let $\mathcal K$ denote the closed cone of nonnegative functions which are even and nonincreasing in $|t|$ on almost every fiber. The heat representation
\[\mathcal A^{-1}f=\int_0^\infty e^{-\La \tau}
e^{\tau\partial_t^2}e^{\tau\Delta_\omega}f\rd \tau\]
shows that $T$ preserves $\mathcal K$. Indeed, the angular heat operator is a positive average, and convolution in $t$ with the Gaussian preserves even decreasing nonnegative functions. For $f\in L^{q'}$, truncation in value and axial support followed by the inverse bound justifies the representation and passage to the limit. Thus $u_0\in\mathcal K$.
Energy-bounded subsets of $\mathcal K$ are precompact in $L^q$. To see this, evenness and monotonicity give
\[2|t|\int_{\mathbb S^{d-1}}w(t,\omega)^2\rd\omega\le\nm{w}_2^2.\]
For the slabs
\[
Q_k=(k,k+1)\times\S^{d-1},\qquad
Q_k'=(k-1,k+2)\times\S^{d-1},
\]
it follows that
\[
\sup_{|k|\ge R}\int_{Q_k}w^2\rd\nu\le\frac CR
\]
uniformly on any energy-bounded set. If $d\ge3$, the local Sobolev inequality and H\"older's inequality imply
\[\int_{Q_k}|w|^{2+4/d}\rd\nu
\le C\left(\int_{Q_k}w^2\rd\nu\right)^{2/d}
\nm{w}_{H^1(Q_k')}^2.\]
The enlarged slabs have bounded overlap, so the $L^{2+4/d}$ tails vanish uniformly. Interpolation with the bounded $L^2$ and $L^{2^*}$ norms gives uniformly vanishing $L^q$ tails for every $2<q<2^*$. If $d=2$, choose $P>\max\{q,2\}$. The analogous estimate is
\[\int_{Q_k}|w|^{4-4/P}\rd\nu
\le C\left(\int_{Q_k}w^2\rd\nu\right)^{1-2/P}
\nm{w}_{H^1(Q_k')}^2,\]
and interpolation gives the same conclusion. 
Rellich compactness on bounded cylinders completes the proof of the claim. 

\noindent\textit{Convergence and the terminal point of the path.}
Starting from $u_0$, define
\[
u_{n+1}=T(u_n),\qquad e_n=E_\La(u_n),\qquad
z_n=\mathcal A^{-1}(u_n^{q-1}).
\]
Then
\[
u_n\in\mathcal K,\qquad \nm{u_n}_q=1,\qquad
e_n\downarrow e_\infty\ge S_{\La,q}>0.
\]
The inequality \eqref{10-29-1} gives 
\[\frac1{e_n}\le E_\La(z_n)\le\frac1{e_{n+1}},\qquad
\nm{u_n-e_nz_n}_\La^2
\le e_n\left(\frac{e_n}{e_{n+1}}-1\right)\longrightarrow0.\]
Cone compactness gives strongly convergent $L^q$ subsequences. Continuity of $\mathcal A^{-1}$ and the last estimate upgrade their convergence to $H^1$. Every limit $v\in\mathcal K$ has $\nm{v}_q=1$ and satisfies
\[\mathcal Av=e_\infty v^{q-1}.\]
Elliptic regularity and the strong maximum principle give $v>0$. The rigidity of all positive finite-energy solutions, after amplitude rescaling, follows from \cite[Corollary 1.3]{Dolbeault etc=2016}; hence $v$ is a translate and multiple of $V$. Membership in $\mathcal K$ fixes the translation at zero and the $L^q$ normalization fixes the amplitude. Thus every cluster point is $V/\nm{V}_q$, and the entire sequence converges to it in $H^1$.

Concatenate the paths $\gamma_{u_n}$ on successive intervals with geometrically decreasing lengths. Their energies are at most $e_n\le e_0$. Their unnormalized points obey
\[\sup_{0\le \tau\le1}\left\|(1-\tau)u_n+\tau e_nz_n-\frac{V}{\nm{V}_q}\right\|_\La
\le\left\|u_n-\frac{V}{\nm{V}_q}\right\|_\La
+\nm{e_nz_n-u_n}_\La\longrightarrow0.\]
Their $L^q$ norms therefore converge uniformly to one, which proves continuity at the terminal point after normalization. Concatenating with the initial path to $u_0$ proves the lemma.
\end{proof}

\subsection{The global bound for nonnegative functions}
\begin{proposition}\label{prop:local-global}
For every $0<\delta<1$,
\be\label{eq:nonnegative-global}
c_{\La,q}^{+}\ge\delta I_{\La,q}(\delta).
\ee
\end{proposition}
\begin{proof}
By homogeneity, take $u\ge0$ with $\nm{u}_q=1$ and $u\notin\MR$. If $\rho(u)\le\delta$, its deficit ratio is at least $I_{\La,q}(\delta)$, which is a stronger estimate. Otherwise, use Lemma \ref{lem:energy-path}. Distance to a fixed set is Lipschitz in the energy norm, and along that path $E_\La\ge S_{\La,q}>0$. Hence $\rho$ is continuous and there is a point $w$ with
\[\rho(w)=\delta,\qquad \nm{w}_q=1,\qquad E_\La(w)\le E_\La(u).\]
Since $\dist_\La(u,\MR)^2\le E_\La(u)$ and the deficit is nonnegative,
\[\begin{aligned}
\frac{D_{\La,q}(u)}{\dist_\La(u,\MR)^2}
&\ge1-\frac{S_{\La,q}}{E_\La(u)}
\ge1-\frac{S_{\La,q}}{E_\La(w)}\\
&=\delta\frac{D_{\La,q}(w)}{\dist_\La(w,\MR)^2}
\ge\delta I_{\La,q}(\delta).
\end{aligned}\]
This proves \eqref{eq:nonnegative-global}.
\end{proof}

\subsection{Real-valued functions and the dimensional bound}
\begin{proposition}\label{prop:sign-reduction}
The optimal stability constants satisfy
\be\label{eq:sign-reduction}
c_{\La,q}\ge\min\left\{\frac12c_{\La,q}^{+},\ 1-2^{2/q-1}\right\}.
\ee
\end{proposition}
\begin{proof}
This is the mass-splitting argument of \cite{Dolbeault etc=2025}. Normalize $\nm{u}_q=1$, write $u=u_+-u_-$, and change its sign if necessary so that
\[m:=\nm{u_-}_q^q\in[0,1/2].\]
The case $m=0$ follows from the nonnegative estimate. For $m>0$, put
\[H(m):=m^{2/q}+(1-m)^{2/q}-1.\]
The positive and negative parts have disjoint supports, so both the gradient and the potential energy split exactly. Therefore
\[D_{\La,q}(u)=D_{\La,q}(u_+)+D_{\La,q}(u_-)+S_{\La,q}H(m).\]
Concavity on $[0,1/2]$ gives
\[
H(m)\ge2\bigl(2^{1-2/q}-1\bigr)m,\qquad
(1-m)^{2/q}\ge1-2\bigl(1-2^{-2/q}\bigr)m.
\]
Consequently,
\[H(m)\ge m^{2/q}-2\bigl(1-2^{-2/q}\bigr)m
\ge m^{2/q}-\frac{1-2^{-2/q}}{2^{1-2/q}-1}H(m),\]
and hence
\[H(m)\ge\zeta_q m^{2/q},\qquad
\zeta_q:=2\bigl(1-2^{2/q-1}\bigr)\in(0,1).\]
Using the sharp inequality for $u_-$, we obtain
\[\begin{aligned}
D_{\La,q}(u_-)+S_{\La,q}H(m)
&\ge E_\La(u_-)-(1-\zeta_q)S_{\La,q}m^{2/q}\ge\zeta_q E_\La(u_-).
\end{aligned}\]
Let $g\in\MR$ be a nearest optimizer for $u_+$, given by Lemma \ref{lem:nearest-relative}. The Hilbert norm inequality yields
\[\begin{aligned}
D_{\La,q}(u)
&\ge c_{\La,q}^{+}E_\La(u_+-g)+\zeta_qE_\La(u_-)\\
&\ge\frac12\min\{c_{\La,q}^{+},\zeta_q\}\,E_\La(u-g)\\
&\ge\min\left\{\frac12c_{\La,q}^{+},1-2^{2/q-1}\right\}\dist_\La(u,\MR)^2.
\end{aligned}\]
Homogeneity gives the result without the normalization.
\end{proof}

\begin{proof}[Proof of Theorem \ref{thm:main}]
Fix $a,b$ as in the theorem. Theorem \ref{thm:local} supplies the same positive constants $\ve_0,\ti\delta$ in every admissible dimension. Let $\delta_0$ be given by \eqref{eq:relative-radius}. Lemma \ref{lem:nearest-relative} and Propositions \ref{prop:local-global}--\ref{prop:sign-reduction} imply
\be\label{eq:global-before-dimension}
c_{\La,q}\ge\min\left\{\frac{\delta_0(q-2)\ve_0}{2},\ 1-2^{2/q-1}\right\}.
\ee
The defining relation for $q$ gives
\[q-2=\frac{4(1+a-b)}{d-2(1+a-b)}\ge\frac{4(1+a-b)}d,\qquad
\frac2q-1=-\frac{2(1+a-b)}d.\]
Since $0<2(1+a-b)/d<1$, concavity of $x\mapsto1-2^{-x}$ on $[0,1]$ gives
\[1-2^{2/q-1}
=1-2^{-2(1+a-b)/d}\ge\frac{1+a-b}{d}.\]
Thus \eqref{eq:global-before-dimension} yields
\[c_{\La,q}\ge\frac{1+a-b}{d}\min\{2\delta_0\ve_0,1\}.
\]
The choice \eqref{eq:local-epsilon-choice} gives $\ve_0\le1/6$, so $2\delta_0\ve_0<1$. We may therefore take
\be\label{eq:main-stability-constant}
\displaystyle
\beta_{a,b}:=\frac{2(1+a-b)\ti\delta\ve_0}{1+\ti\delta}>0,
\ee
with $\ve_0$ and $\ti\delta$ given by \eqref{eq:local-epsilon-choice} and \eqref{eq:local-delta-choice}.
This number depends only on $a,b$. Finally, the Emden--Fowler identities \eqref{equiv} preserve both the deficit and the full energy distance to the optimizer manifold. Hence $c_{d,a,b}=c_{\La,q}$ by \eqref{CKN-C-c}, proving the theorem in all admissible dimensions $d\ge2$.
\end{proof}

\section{The critical case $a=b>0$}\label{sec:critical}

The critical case has a different dimensional scale from Theorem \ref{thm:main}. Its proof combines Sobolev stability with the angular energy introduced by the radial power transformation. Throughout this section,
\[
d\ge3,\qquad 0<a<a_c,\qquad q=2^*.
\]

\subsection{Angular energy and the center of a Sobolev optimizer}
Use the angular energy $A(F)$ from \eqref{eq:angular-energy} and the Sobolev deficit from \eqref{eq:sobolev-deficit}. For $F\in D^{1,2}(\R^d)$, write
\[
 E(F):=\int_{\R^d}|\nabla F|^2\rd y.
\]
Then
\[
0\le A(F)\le E(F),\qquad A(F)=0\quad\text{if }F\text{ is radial}.
\]
Define the Sobolev bubbles by
\[
 U_{\lambda,z}(y):=\lambda^{-a_c}\left(1+\left|\frac{y-z}{\lambda}\right|^2\right)^{-a_c},\qquad
 \lambda>0,\quad z\in\R^d.
\]
Explicitly, the full Sobolev optimizer manifold and its centered submanifold are
\[
 \mathcal B=\{cU_{\lambda,z}:c\in\R,\ \lambda>0,\ z\in\R^d\},\qquad
 \mathcal B_0=\{cU_{\lambda,0}:c\in\R,\ \lambda>0\}.
\]
We include zero in both manifolds and set
\[
 e(F):=\dist_{D^{1,2}}(F,\mathcal B),\qquad
 r(F):=\dist_{D^{1,2}}(F,\mathcal B_0).
\]
The Sobolev stability estimate \cite[Theorem 1.1]{Dolbeault etc=2025}, with the constant $\beta_{\mathrm{Sob}}$ from \eqref{eq:intro-sobolev-dimension}, reads
\be\label{eq:critical-sobolev-stability}
 D_{\mathrm{Sob}}(F)\ge\frac{\beta_{\mathrm{Sob}}}{d}\,e(F)^2
 \qquad\text{for every }d\ge3.
\ee

\begin{lemma}\label{lem:critical-center-control}
For every $d\ge3$ and $F\in D^{1,2}(\R^d)$,
\be\label{eq:critical-center-control}
 r(F)^2\le25d\,e(F)^2+24d\,A(F).
\ee
\end{lemma}
\begin{proof}
We first prove the estimate
\be\label{eq:critical-bubble-center}
 r(B)^2\le6d\,A(B)\qquad\text{for every }B\in\mathcal B.
\ee
By homogeneity, dilation invariance and rotation invariance, it suffices to consider
\[
 U_z(y)=(1+|y-z|^2)^{-a_c},\qquad z=t e_1,\quad t\ge0.
\]
Writing $y=z+\rho\omega$, we obtain
\[
 \nabla U_z(y)=-2a_c\rho\omega(1+\rho^2)^{-a_c-1}.
\]
Projection onto the tangent space to the sphere centered at the origin therefore gives
\be\label{eq:critical-bubble-angular-integral}
 \frac{A(U_z)}{E(U_0)}
 =\int_0^\infty\int_{\S^{d-1}}
 \frac{t^2(1-\omega_1^2)}{\rho^2+t^2+2\rho t\omega_1}
 \rd\sigma(\omega)\rd\mu_d(\rho),
\ee
where $\sigma$ is normalized surface measure and
\[
 \rd\mu_d(\rho):=
 \frac{\rho^{d+1}(1+\rho^2)^{-d}\rd\rho}
 {\displaystyle\int_0^\infty r^{d+1}(1+r^2)^{-d}\rd r}.
\]
The inequalities
\[
 \rho^2+t^2+2\rho t\omega_1\le2(\rho^2+t^2),\qquad
 \frac{t^2}{\rho^2+t^2}\ge\frac{\min\{t^2,1\}}{1+\rho^2}
\]
and the integral identities
\[
 \int_{\S^{d-1}}(1-\omega_1^2)\rd\sigma=\frac{d-1}{d},\qquad
 \int_0^\infty\frac{\rd\mu_d(\rho)}{1+\rho^2}=\frac{d-2}{2d}
\]
imply
\be\label{eq:critical-bubble-angular-lower}
 \frac{A(U_z)}{E(U_0)}
 \ge\frac{(d-1)(d-2)}{4d^2}\min\{t^2,1\}
 \ge\frac1{18}\min\{t^2,1\}.
\ee
On the other hand, comparison with $0$ and $U_0$, followed by integration along the translation path, gives
\[
 r(U_z)^2\le\min\left\{E(U_0),\ t^2\nm{\nabla\pa_1U_0}_2^2\right\}.
\]
Differentiating
\[
-\Delta U_0=d(d-2)U_0^{q-1},
\]
and using
\[
 \frac{\displaystyle\int_0^\infty\rho^{d+1}(1+\rho^2)^{-d-2}\rd\rho}
 {\displaystyle\int_0^\infty\rho^{d+1}(1+\rho^2)^{-d}\rd\rho}
 =\frac{d-2}{4(d+1)},
\]
we get
\[
 \frac{\nm{\nabla\pa_1U_0}_2^2}{E(U_0)}
 =\frac{d^2-4}{4(d+1)}\le\frac d4.
\]
Consequently, since $d\ge3$,
\[
 \frac{r(U_z)^2}{E(U_0)}
 \le\min\{1,dt^2/4\}
 \le\frac d3\min\{1,t^2\}.
\]
Together with \eqref{eq:critical-bubble-angular-lower}, this proves \eqref{eq:critical-bubble-center}.

For arbitrary $F\in D^{1,2}(\R^d)$ and $B\in\mathcal B$, the triangle inequality for the distance and \eqref{eq:critical-bubble-center} give
\[
 \begin{aligned}
 r(F)^2
 &\le2E(F-B)+2r(B)^2\\
 &\le2E(F-B)+12d\,A(B)\\
 &\le(2+24d)E(F-B)+24d\,A(F).
 \end{aligned}
\]
Here we used
\[
A(B)\le2A(F)+2A(F-B)\le2A(F)+2E(F-B).
\]
Taking the infimum over $B\in\mathcal B$ and using $2+24d\le25d$ proves \eqref{eq:critical-center-control}. 
\end{proof}

\subsection{A power transformation and the lower bound}
With $\alpha$ from \eqref{eq:critical-parameters}, apply the radial change of variables of \cite[Section~3.1, equation~(3.1)]{Dolbeault etc=2016}, specialized to $b=a$:
\be\label{eq:critical-power-change}
 y=|x|^{\alpha-1}x,\qquad f(x)=F(y).
\ee
The identities
\[
d-2-2a=\alpha(d-2),\qquad d-aq=\alpha d
\]
and polar coordinates give
\be\label{eq:critical-power-identities}
 \begin{aligned}
 E_a(f)&=\alpha\bigl(E(F)+(\alpha^{-2}-1) A(F)\bigr),\\
 \int_{\R^d}|x|^{-aq}|f|^q\rd x
 &=\alpha^{-1}\int_{\R^d}|F|^q\rd y.
 \end{aligned}
\ee
Since
\[
E(F)\le E(F)+(\alpha^{-2}-1) A(F)\le\alpha^{-2}E(F),
\]
the transformation extends by density to an isomorphism between $D_a^{1,2}(\R^d)$ and $D^{1,2}(\R^d)$.
The transformation \eqref{eq:critical-power-change} maps $\MR_{a,a}$ onto $\mathcal B_0$. Indeed,
\[
(a_c-a)(q-2)=2\alpha,
\]
so the profile \eqref{W} becomes a multiple of $(1+|y|^2)^{-a_c}$, and dilations and amplitudes produce precisely $\mathcal B_0$. 
The Sobolev inequality, with equality on centered bubbles, now gives
\be\label{eq:critical-sharp-constant}
 S_{d,a,a}=\alpha^{1+2/q}S_d.
\ee
Therefore the weighted deficit satisfies
\be\label{eq:critical-deficit-identity}
 D_{d,a,a}(f)=\alpha\bigl(D_{\mathrm{Sob}}(F)+(\alpha^{-2}-1) A(F)\bigr).
\ee
For every radial $G$,
\[
A(F-G)=A(F).
\]
Therefore the distance has the exact expression
\be\label{eq:critical-distance-identity}
 \dist_{D_a^{1,2}}(f,\MR_{a,a})^2
 =\alpha\bigl(r(F)^2+(\alpha^{-2}-1) A(F)\bigr).
\ee

Combining \eqref{eq:critical-sobolev-stability} and Lemma \ref{lem:critical-center-control} with these identities gives
\[
 \begin{aligned}
 D_{d,a,a}(f)&\ge\alpha\left(\frac{\beta_{\mathrm{Sob}}}{d}e(F)^2+(\alpha^{-2}-1) A(F)\right),\\
 \dist_{D_a^{1,2}}(f,\MR_{a,a})^2
 &\le\alpha\left(25d\,e(F)^2+(24d+\alpha^{-2}-1)A(F)\right).
 \end{aligned}
\]
It follows, for arbitrary real-valued $f$, that
\be\label{eq:critical-explicit-lower}
 c_{d,a,a}\ge\min\left\{\frac{\beta_{\mathrm{Sob}}}{25d^2},\ \frac{\alpha^{-2}-1}{24d+\alpha^{-2}-1}\right\}.
\ee
Finally,
\[
 \alpha^{-2}-1=\left(1-\frac{2a}{d-2}\right)^{-2}-1
 \ge\frac{4a}{d-2}\ge\frac{4a}{d}.
\]
Because $4a<2(d-2)<d^2$, monotonicity of $x\mapsto x/(24d+x)$ yields
\[
 \frac{\alpha^{-2}-1}{24d+\alpha^{-2}-1}
 \ge\frac{4a}{24d^2+4a}
 \ge\frac{4a}{25d^2}.
\]
Thus, with $\beta_a$ as in \eqref{eq:critical-beta},
\be\label{eq:critical-dimensional-lower}
 c_{d,a,a}\ge\frac{\beta_a}{d^2}.
\ee

\subsection{The upper bound and the optimal dimensional scale}
\begin{proof}[Completion of the proof of Theorem \ref{thm:critical}]
It remains to establish the upper bound. Work on the cylinder with $\La,s$ as in \eqref{eq:mass-parameter}. The first angular eigenvalue and a corresponding eigenfunction are
\[
 \mu_{1,0}=\frac{s(2s+q-2)}q,\qquad
 \phi(t,\omega)=\sbr{\cosh\left(\frac{q-2}{2}\sqrt\La\,t\right)}^{-2s/(q-2)}\omega_1.
\]
Differentiating the formula for $\phi$ and using
\[
-\Delta_\omega\omega_1=(d-1)\omega_1
\]
gives
\[
 (-\pa_t^2-\Delta_\omega+\La)\phi=\mu_{1,0}V^{q-2}\phi.
\]
Since $s>1$,
\[
\phi\in H^1(\CR),\qquad \phi/V\in L^\infty(\CR).
\]
Thus the perturbation
\[
u_\varepsilon=V+\varepsilon\phi
\]
is positive for sufficiently small real $\varepsilon$.

The zero spherical mean of $\omega_1$ shows that $\phi$ is orthogonal in the energy inner product to every function of $t$ alone. Hence, for each $g\in\MR$,
\[
 \nm{u_\varepsilon-g}_\La^2
 =\nm{V-g}_\La^2+\varepsilon^2\nm\phi_\La^2,
\]
so that
\be\label{eq:critical-perturbation-distance}
 \dist_\La(u_\varepsilon,\MR)^2=\varepsilon^2\nm\phi_\La^2.
\ee
The boundedness of $\phi/V$ justifies the Taylor expansion of the deficit at $V$. Using
\[
S_{\La,q}\nm V_q^{2-q}=1,\qquad
\int_\CR V^{q-1}\phi\rd\nu=0,
\]
we obtain, in each fixed dimension,
\[
 \begin{aligned}
 \nm{u_\varepsilon}_\La^2-S_{\La,q}\nm{u_\varepsilon}_q^2
 &=\varepsilon^2\left(\nm\phi_\La^2-(q-1)\int_\CR V^{q-2}\phi^2\rd\nu\right)+o(\varepsilon^2)\\
 &=\varepsilon^2\nm\phi_\La^2\left(1-\frac{q-1}{\mu_{1,0}}\right)+o(\varepsilon^2).
 \end{aligned}
\]
By \eqref{CKN-C-c}, \eqref{eq:critical-perturbation-distance}, and the definition of the optimal stability constant,
\be\label{eq:critical-explicit-upper}
 c_{d,a,a}\le1-\frac{q-1}{\mu_{1,0}}
 =\frac{(2s-q)(s+q-1)}{s(2s+q-2)}.
\ee
For $a>0$, one has $s>d/(d-2)=q/2$, so this upper bound is positive.

Now fix $a>0$ and let $d\to\infty$ through the admissible dimensions. Direct expansion gives
\[
 q=2+\frac4d+\frac8{d^2}+O(d^{-3}),\qquad
 s=1+\frac2d+\frac{8a+4}{d^2}+O_a(d^{-3}).
\]
Thus
\[
2s-q=\frac{16a}{d^2}+O_a(d^{-3}),
\]
while
\[
 \frac{s+q-1}{s(2s+q-2)}=1+O_a(d^{-1}).
\]
Consequently \eqref{eq:critical-explicit-upper} implies
\be\label{eq:critical-asymptotic-calculation}
 c_{d,a,a}\le\frac{16a}{d^2}+O_a(d^{-3}).
\ee
Together with \eqref{eq:critical-explicit-lower} and \eqref{eq:critical-dimensional-lower}, this proves all assertions of Theorem \ref{thm:critical}. In particular, the dimensional order is $d^{-2}$.
\end{proof}

\section{The degenerate case on the Felli--Schneider curve}\label{sec:degenerate}
Throughout this section $a<0$ is fixed and $b=b_{FS}(a)$. We use $n$ from \eqref{eq:fs-parameters}, the operator $\mathcal A$ from Section~\ref{sec:global}, and the normalized measure $\nu_1$ from Section~\ref{sec:local}. The cylinder formulation and quartic stability framework are those of \cite[Sections~1.3--1.4]{Frank-Peteranderl=2024}. We track the constants as the dimension varies. Write
\be\label{eq:fs-normalized-functionals}
 \mathsf E(u):=\nm{u}_{\La,\nu_1}^2,
 \qquad \mathsf D(u):=\mathsf E(u)-\nm{u}_{q,\nu_1}^2,
 \qquad \|\ell\|_*:=\sup_{\mathsf E(u)=1}|\ell(u)|.
\ee
Thus \eqref{CKN-C-1} gives $\mathsf E(u)\ge\nm{u}_{q,\nu_1}^2$ and $\mathsf E(V)=\int V^q\rd\nu_1=1$. 
The FS relation and its large-dimensional consequences are
\be\label{eq:fs-parameter-relations}
 \frac{\La}{n^2}=\frac{d-1}{2n+1},\qquad
 n=\frac d2-1-2a+O_a(d^{-1}),\qquad
 2n+2-d\longrightarrow-4a,\qquad \theta\asymp_a d^{-1}.
\ee
Here and below $c_a,C_a>0$ may change from line to line and depend only on $a$. We first take $d\ge d_0(a)$, increasing this threshold $d_0(a)$ as necessary, so that $d\ge6$, $n\ge2$, and $1\le(2n+2)/d\le2$. Finitely many smaller dimensions will be included at the end.

\subsection{Approximate profiles and a uniform local chart}
The projection framework of \cite[Propositions~2--3]{Frank-Peteranderl=2024} separates the amplitude and axial-translation directions
from the additional zero modes on the FS curve. 
We develop this framework using the following nonlinear family. 
The following Lemmas~\ref{lem:fs-chart}--\ref{lem:fs-forms} reformulate \cite[Propositions~2--3]{Frank-Peteranderl=2024} by establishing a local modulation decomposition together with profile and coercivity estimates, with constants and neighborhood radii depending only on $a$.

For $\sigma\in\R$ and $\tau\in\R^d$ in a fixed small ball, set
\be\label{eq:fs-profiles}
 \Phi_{\sigma,\tau}(t,\omega)
 :=V(0)\left[\cosh\left(\frac{\sqrt\La}{n}
                  \left(t-\frac\sigma{\sqrt d}\right)\right)
              -\frac{\tau\cdot\omega}{\sqrt d}\right]^{-n},
 \qquad
 \phi_{\sigma,\tau}:=\frac{\Phi_{\sigma,\tau}}
 {\nm{\Phi_{\sigma,\tau}}_{q,\nu_1}}.
\ee
The tangent space to $(c,\sigma,\tau)\mapsto c\phi_{\sigma,\tau}$ at $(1,\sigma,\tau)$ is
\be\label{eq:fs-tangent-space}
 \mathcal T_{\sigma,\tau}
 :=\operatorname{span}\left\{
 \phi_{\sigma,\tau},\ \partial_\sigma\phi_{\sigma,\tau},\quad
 \partial_{\tau_1}\phi_{\sigma,\tau},\ldots,
 \partial_{\tau_d}\phi_{\sigma,\tau}\right\}.
\ee
At $(\sigma,\tau)=(0,0)$ this space is the full Hessian kernel identified in \cite[Section~3.1, equation~(3.1)]{Frank-Peteranderl=2024}: the parameter derivatives are proportional to $V$, $V'$, and $V^{q/2}\omega_j$, respectively.
For the ordered parameter vector
\[
 \mathbf p=(p_0,\ldots,p_{d+1})=(c,\sigma,\tau_1,\ldots,\tau_d),
\]
the associated $(d+2)\times(d+2)$ energy Gram matrix is
\be\label{eq:fs-gram-matrix}
 G(\mathbf p):=
 \left(\left\langle
 \partial_{p_i}(c\phi_{\sigma,\tau}),
 \partial_{p_j}(c\phi_{\sigma,\tau})
 \right\rangle_{\La,\nu_1}\right)_{i,j=0}^{d+1}.
\ee
In particular, for $\xi=(\xi_0,\ldots,\xi_{d+1})\in\R^{d+2}$,
\[
 \xi^{\mathsf T}G(\mathbf p)\xi
 =\mathsf E\left(
 \xi_0\phi_{\sigma,\tau}
 +c\xi_1\partial_\sigma\phi_{\sigma,\tau}
 +c\sum_{j=1}^d\xi_{j+1}\partial_{\tau_j}\phi_{\sigma,\tau}
 \right).
\]
Thus $G(\mathbf p)$ is invertible exactly when these parameter derivatives are linearly independent. For $c\ne0$, this is equivalent to $\dim\mathcal T_{\sigma,\tau}=d+2$. Uniform invertibility on a parameter neighborhood means that there is $g_a>0$, depending only on $a$ and independent of $d$, such that throughout that neighborhood
\be\label{eq:fs-uniform-gram-inverse}
 \xi^{\mathsf T}G(\mathbf p)\xi\ge g_a|\xi|^2
 \quad\text{for all }\xi\in\R^{d+2},
 \qquad\text{equivalently}\qquad
 \|G(\mathbf p)^{-1}\|_{\rm op}\le g_a^{-1}.
\ee
Here $|\xi|$ is the Euclidean norm and $\|\cdot\|_{\rm op}$ is its induced operator norm. The lower bound controls the size of a parameter variation by the energy norm of the corresponding first-order profile variation.

Axial translation reduces all subsequent estimates to $\sigma=0$; write $\phi=\phi_{0,\tau}$ and
\[
 R:=\mathcal A\phi-\phi^{q-1},\qquad
 \rd\mu_0:=V^q\rd\nu_1,\qquad \rd\mu_\phi:=\phi^q\rd\nu_1,
 \qquad z:=\tanh(\sqrt\La\,t/n),\qquad
 X:=\sqrt{1-z^2}\,\omega.
\]
Both $\mu_0$ and $\mu_\phi$ are probability measures. Under $\mu_0$, the variable $z$ has density proportional to $(1-z^2)^n$ and is independent of $\omega$.

\begin{lemma}\label{lem:fs-chart}
There are $r_a\in(0,1/2)$ and $\varepsilon_a>0$, depending only on $a$ and independent of $d$, such that on the parameter neighborhood
\[
 \left\{(c,\sigma,\tau)\in\R\times\R\times\R^d:
          |c-1|+|\sigma|+|\tau|<r_a\right\},
\]
the first three multilinear derivatives of $(c,\sigma,\tau)\mapsto c\phi_{\sigma,\tau}$ have uniformly bounded energy norms, and $G(\mathbf p)$ satisfies \eqref{eq:fs-uniform-gram-inverse}. Consequently, every $u$ in the fixed energy ball
\be\label{eq:fs-energy-neighborhood}
 \left\{u\in H^1(\CR,\nu_1):
   \mathsf E(u-V)=\nm{u-V}_{\La,\nu_1}^2<\varepsilon_a^2\right\}
\ee
admits a decomposition
\be\label{eq:fs-chart-decomposition}
 u=c\phi_{\sigma,\tau}+h,\qquad
 h\perp_{\La,\nu_1}\mathcal T_{\sigma,\tau},\qquad
 |c-1|+|\sigma|+|\tau|+\sqrt{\mathsf E(h)}
 \le C_a\sqrt{\mathsf E(u-V)}.
\ee
Moreover, uniformly over the whole energy unit sphere of $\mathcal T_{0,\tau}$,
\be\label{eq:fs-tangent-moment}
 \sup_{\substack{w\in\mathcal T_{0,\tau}\\\mathsf E(w)=1}}
 \int\left|\frac{\mathcal Aw}{\phi^{q-1}}\right|^4\rd\mu_\phi
 \le C_a.
\ee
\end{lemma}
\begin{proof}
Beta integrals and spherical moments show that all unit-direction linear combinations of $\sqrt d\,z$ and $\sqrt d\,X_i$ have uniform sub-Gaussian bounds. Indeed, their even moments are bounded by $C^k(2k-1)!!$, since $2n+2\asymp d$ and
\[
 d^k\frac1{|\S^{d-1}|}\int_{\S^{d-1}}(\omega\cdot e)^{2k}\rd\omega
 \le(2k-1)!!\qquad(|e|=1).
\]
For small fixed $|\tau|$, the denominator $1-\tau\cdot X/\sqrt d$ is at least $1/2$. The density change is bounded by 
    \[C_a\exp(C_a|\tau|\,|\sqrt d X\cdot e|),\] 
and the same moment bounds therefore hold after tilting. Differentiating \eqref{eq:fs-profiles} shows that its first three mixed logarithmic derivatives, in arbitrary unit parameter directions, have uniformly bounded fixed-order moments under the tilted measure.

Direct differentiation, using \eqref{eq:fs-parameter-relations}, gives
\be\label{eq:fs-profile-equation}
 \mathcal A\phi=P_\tau(\omega)\phi^{q-1},\qquad
 P_\tau=F^{1/(n+1)}\left[
 1-\frac{n^2|\tau|^2}{d\La}
 +\frac{n^2/\La-1}{d}(\tau\cdot\omega)^2\right],
\ee
with 
    \[F:=\int\Phi_{0,\tau}^q\rd\nu_1.\]
Differentiating this identity up to order three and using the preceding moments bounds both $D^k\phi/\phi$ and $\mathcal A(D^k\phi)/\phi^{q-1}$ in fixed-order tilted norms. Integration by parts then bounds $\mathsf E(D^k\phi)$, for $k\le3$. These bounds concern entire unit directions and do not require summing over $d$ coordinates.

At $\mathbf p=(1,0,0)$ the energy Gram matrix $G(\mathbf p)$ is diagonal, with entries
\[
 1,\qquad
 \frac{n(n+2)(d-1)}{d(2n+1)(2n+3)},\qquad
 \frac{2n(n+1)(n+2)}{d^2(2n+3)}
 \quad\text{(the last entry repeated $d$ times)}.
\]
They are bounded above and below independently of $d$. The derivative estimates preserve these bounds on a fixed ball. The quantitative implicit function theorem, applied to the energy-orthogonality equations for the residual in \eqref{eq:fs-chart-decomposition}, gives the stated decomposition with a uniform radius. Differentiating \eqref{eq:fs-profile-equation} once and applying the uniform Gram inverse gives \eqref{eq:fs-tangent-moment}; again this estimates a whole tangent direction rather than individual coordinates.
\end{proof}

\begin{lemma}\label{lem:fs-profile}
For $|\tau|\le t_0(a)$ and sufficiently large $d$,
\be\label{eq:fs-profile-estimates}
 c_a\frac{|\tau|^4}{d^3}\le\mathsf D(\phi)
 \le C_a\frac{|\tau|^4}{d^3},\qquad
 \|R\|_*\le C_a\frac{|\tau|^2}{d^3},\qquad
 \frac{\mathcal A\phi}{\phi}>0.
\ee
The estimates are uniform as $\tau\to0$, and
\be\label{eq:fs-density-bound}
 \int\left(\frac{\rd\mu_0}{\rd\mu_\phi}\right)^2\rd\mu_\phi\le C_a.
\ee
\end{lemma}
\begin{proof}
Rotate so that $\tau=|\tau|e_d$ and put 
    \[\zeta=|\tau|^2/d \quad\text{and} \quad \chi=n+1-d/2.\] 
Then $\chi\to-2a>0$. The binomial series and the beta and spherical moment identities recorded in \cite[equation~(4.6)]{Frank-Peteranderl=2024} give the following formulas; see also \eqref{eq:fs-trial-moments} below with $\varepsilon=|\tau|/\sqrt d$:
\[
 e(\zeta):=\mathsf E(\Phi_{0,\tau})
 =\leftidx{_2}F_1(n,n+1;d/2;\zeta),\qquad
 F:=\int\Phi_{0,\tau}^q\rd\nu_1=\leftidx{_2}F_1(n+1,n+1;d/2;\zeta).
\]
Euler's transformation yields
\[
 e=(1-\zeta)^{-d/2-2\chi+1}G_1,\qquad 
 F=(1-\zeta)^{-d/2-2\chi}G,
\]
with
    \[G(\zeta):=\leftidx{_2}F_1(-\chi,-\chi;d/2;\zeta),\qquad  G_1:=G-\zeta G'/\chi.\]
The coefficients of $G$ are nonnegative. Gauss's formula for $G(1)$ and $G'(1)$ bounds the sums of the coefficients of degree at least two, including their degree-weighted sum, by $C_ad^{-2}$, uniformly for $\chi$ in the compact interval arising here. Therefore
\[
 \log G=\frac{2\chi^2}{d}\zeta+O_a(\zeta^2/d^2),\qquad
 \log G_1=\frac{2\chi(\chi-1)}d\zeta+O_a(\zeta^2/d^2).
\]
The linear terms cancel exactly in the normalized energy:
\be\label{eq:fs-profile-log-energy}
 \log\mathsf E(\phi)
 =\log e-\frac n{n+1}\log F
 =\frac\chi{n+1}\bigl[-\log(1-\zeta)-\zeta\bigr]
   +O_a(\zeta^2/d^2).
\ee
Since $\mathsf D(\phi)=\mathsf E(\phi)-1$, this proves the two-sided deficit estimate.

Subtracting the Rayleigh multiplier in \eqref{eq:fs-profile-equation} gives
\[
 \mathcal A\phi-\mathsf E(\phi)\phi^{q-1}
 =F^{1/(n+1)}\left(\frac{n^2}{\La}-1\right)\zeta
 \left(\omega_d^2-\int\omega_d^2\rd\mu_\phi\right)\phi^{q-1}.
\]
Here $n^2/\La-1=O_a(d^{-1})$ and
$\int\omega_d^4\rd\mu_\phi\le C_ad^{-2}$ by the tilted moments above. Since $q'\le2$, the sharp embedding \eqref{CKN-C-1} and H\"older's inequality bound the dual norm of this expression by $C_a|\tau|^2/d^3$. Replacing $\mathsf E(\phi)$ by $1$ costs at most $\mathsf D(\phi)$ in dual norm, proving the residual estimate. Positivity follows directly from \eqref{eq:fs-profile-equation} for small fixed $|\tau|$.

Finally, the left side of \eqref{eq:fs-density-bound} equals
\[
 F\int(1-\tau\cdot X/\sqrt d)^{2(n+1)}\rd\mu_0.
\]
Both factors are bounded by the same exponential-moment estimates used in Lemma~\ref{lem:fs-chart}. This proves the density comparison.
\end{proof}

\subsection{Uniform transverse forms and angular spectral cutoffs}

Coercivity on the orthogonal complement of the full Hessian kernel underlies the second projection in \cite[Proposition~3 and Section~3.2]{Frank-Peteranderl=2024}. We quantify this coercivity for the approximate profiles, keeping the dimension dependence explicit.
For the normalized profile $\phi=\phi_{0,\tau}$, set
\[
 B_\phi(v,w):=\langle v,w\rangle_{\La,\nu_1}
 -(q-1)\int_\CR\phi^\theta vw\rd\nu_1.
\]
All constants in the next three lemmas are independent of $d$.
The parameter radius may be decreased and the lower dimension threshold
increased a finite number of times.

\begin{lemma}\label{lem:fs-forms}
There are $b_0,B_0>0$, depending only on $a$, such that, on a fixed
sufficiently small parameter ball,
\begin{align}
 B_\phi(v,v)&\ge b_0\theta\mathsf E(v)
 &&\text{if }v\perp_{\La,\nu_1}\mathcal T_{0,\tau},
 \label{eq:fs-transverse-form}\\
 \|B_\phi(w,\cdot)\|_*&\le B_0\theta\mathsf E(w)^{1/2}
 &&\text{if }w\in\mathcal T_{0,\tau}.
 \label{eq:fs-tangent-form}
\end{align}
For every $v=Vf$, write
\[
 H^2:=\int_\CR f^2\rd\mu_0,\qquad
 G^2:=\int_\CR\left((1-z^2)|\partial_zf|^2
 +\frac{n^2}{\La(1-z^2)}|\nabla_\omega f|^2\right)\rd\mu_0.
\]
Then
\begin{equation}\label{eq:fs-form-comparison}
 \mathsf E(v)=H^2+\frac{G^2}{n(n+1)},\qquad
 \int_\CR\phi^\theta v^2\rd\nu_1
 \le H^2+C_a|\tau|\left(\frac{G^2}{n(n+1)}+\theta H^2\right).
\end{equation}
\end{lemma}

\begin{proof}
The energy identity follows from the ground-state transformation at $V$.
The positive diffusion associated with the Dirichlet form $G^2$ has
eigenfunctions $z$ and $X\cdot e$, for every unit vector $e\in\R^d$,
at eigenvalue $2n+2$. The squared gradient of $X\cdot e$ in this form is
bounded by $n^2/\La\le3$. Integration by parts and Cauchy--Schwarz give
\[
 \left|\int_\CR(X\cdot e)f^2\rd\mu_0\right|
 \le\frac{2\sqrt3}{2n+2}HG.
\]
Recall from \eqref{eq:fs-profile-equation} that $F\ge1$ by symmetry
and convexity. The formula for the profile therefore yields
\[
 \frac{\phi^\theta}{V^\theta}
 =F^{-1/(n+1)}\left(1-\frac{\tau\cdot X}{\sqrt d}\right)^{-2}
 \le1+\frac{2\tau\cdot X}{\sqrt d}
       +\frac{C|\tau|^2}{d}.
\]
Consequently,
\[
 \int_\CR\phi^\theta v^2\rd\nu_1-H^2
 \le C\left(\frac{|\tau|}{\sqrt d(2n+2)}HG
                         +\frac{|\tau|^2}{d}H^2\right).
\]
Since $n\asymp_a d$ and $\theta=2/n$, Young's inequality proves
\eqref{eq:fs-form-comparison}. This averaged estimate retains the
coercivity scale $\theta$.

At $\tau=0$, the kernel of the rank-corrected half Hessian, identified in \cite[Section~3.1, equation~(3.1)]{Frank-Peteranderl=2024}, takes the following form in our variables:
\[
 L:=\operatorname{span}\{V,Vz,VX_1,\ldots,VX_d\}.
\]
Separation of variables gives the next gap in the energy metric as
$2/(n+3)$. Indeed, the eigenvalue of the radial branch $(i,j)=(0,2)$ is
$(n+2)(n+3)/(n(n+1))$, and every remaining angular mode has a
spectral parameter at least $n+2$. The uniform chart in
Lemma~\ref{lem:fs-chart} shows that the angle between $L$ and
$\mathcal T_{0,\tau}$ is at most $C_a|\tau|$ in operator norm.
Thus, if $v\perp_{\La,\nu_1}\mathcal T_{0,\tau}$ and $P_L$ denotes
energy projection onto $L$,
\[
 \mathsf E(P_Lv)\le C_a|\tau|^2\mathsf E(v),\qquad
 |\langle v,V\rangle_{\La,\nu_1}|^2\le\mathsf E(P_Lv).
\]
The form without the rank correction therefore satisfies
\begin{align*}
 B_V(v,v)
 &\ge\frac2{n+3}\bigl(\mathsf E(v)-\mathsf E(P_Lv)\bigr)
       -\theta|\langle v,V\rangle_{\La,\nu_1}|^2\\
 &\ge\left(\frac2{n+3}
       -C_a|\tau|^2\left(\frac2{n+3}+\theta\right)\right)\mathsf E(v).
\end{align*}
The ratio $(2/(n+3))/\theta=n/(n+3)$ is bounded below by $2/5$.
For a fixed small radius this gives $B_V(v,v)\ge c_a\theta\mathsf E(v)$.
Moreover,
\[
 \theta H^2\le C_a B_V(v,v),\qquad
 \frac{G^2}{n(n+1)}=B_V(v,v)+\theta H^2\le C_aB_V(v,v).
\]
The signed perturbation estimate in \eqref{eq:fs-form-comparison}
can now be absorbed by decreasing the radius. This proves
\eqref{eq:fs-transverse-form} for $B_\phi$ itself.

For \eqref{eq:fs-tangent-form}, differentiate the profile equation
\eqref{eq:fs-profile-equation}, written
$\mathcal A\phi=P_\tau\phi^{q-1}$.
Its explicit coefficient satisfies
\[
 \|P_\tau-1\|_\infty\le C_a\theta|\tau|^2,
 \qquad \|D_eP_\tau\|_{L^r(\mu_\phi)}\le C_{a,r}\theta
 \quad(|e|=1)
\]
for every fixed finite $r$. The axial derivative of $P_\tau$ is zero.
For a parameter derivative $w=D_e\phi$, the differentiated equation gives
\[
 B_\phi(w,v)=\int_\CR
 \bigl[(D_eP_\tau)\phi^{q-1}
       +(q-1)(P_\tau-1)\phi^\theta D_e\phi\bigr]v\rd\nu_1.
\]
For the amplitude direction,
$B_\phi(\phi,v)=\int(P_\tau-(q-1))\phi^{q-1}v\rd\nu_1$.
The fixed-order profile moments from Lemma~\ref{lem:fs-chart},
H\"older's inequality, and $\|v\|_{q,\nu_1}^2\le\mathsf E(v)$
bound these functionals by $C_a\theta\mathsf E(v)^{1/2}$ for a
whole unit parameter direction. The uniform Gram inverse in that lemma
passes to every energy-unit vector in $\mathcal T_{0,\tau}$, with
no sum over its $d$ angular directions.
\end{proof}

\begin{lemma}\label{lem:fs-lowmom}
Fix an integer $K\ge1$ and use the angular cutoff $\mu_{K,0}$ of
\eqref{mu-ij}. For all sufficiently large $d$, depending on $a,K$,
there are at most $(K+1)(K+2)$ distinct eigenvalues below this cutoff.
For every entire eigenspace $E_\mu$ with $\mu\le\mu_{K,0}$,
\begin{equation}\label{eq:fs-low-moment}
 \int_\CR V^{q-4}v^4\rd\nu_1
 \le D_K\left(\int_\CR V^\theta v^2\rd\nu_1\right)^2
 \quad(v\in E_\mu),
\end{equation}
where the dimension-independent constant may be chosen as
\[
 D_K:=16^{K+2}\left(1-\left(\frac34\right)^{K+1}\right)^2
             \frac{(2K+3)!}{4^{K+1}((K+1)!)^2}.
\]
In addition,
\begin{equation}\label{eq:fs-angular-gap}
 \mu_{K,0}-1\ge K\theta.
\end{equation}
\end{lemma}

\begin{proof}
We use the separation into radial and spherical harmonic modes employed in \cite[Sections~3.1 and~4.2, equation~(4.8)]{Frank-Peteranderl=2024}, together with the moment estimates from Section~\ref{sec:cylinder}.
At the FS parameters the separation parameters from \eqref{mu-ij}
satisfy
\[
 \eta_i^2=n^2+\frac{n^2}{\La}i(i+d-2),\qquad
 \eta_i^2-(n+i)^2
   =\frac{i(i-1)(2n+2-d)}{d-1}\ge0.
\]
Here $\eta_0=n$ and $\eta_1=n+1$. Since $2n+2-d$ is bounded for
fixed $a$, for every fixed $K$ an increase of the dimension threshold gives
\[
 n+K\le\eta_K\le n+K+1.
\]
Thus
\begin{equation}\label{eq:fs-joint-cutoff}
 \mu_{i,j}\le\mu_{K,0}
 \quad\Longrightarrow\quad
 0\le i\le K,\qquad \eta_i+j\le n+K+1,\qquad i+j\le K+1.
\end{equation}
Counting pairs $(i,j)$ gives the stated upper bound on the number of
distinct eigenvalues; angular multiplicities are excluded from this count.

Lemma~\ref{lem:radial} applies with $M=K+1$ and $\eta=\eta_i$.
Its proof and \eqref{eq:radialratio} require only $\eta\ge n$, so
they apply also to the equality $\eta_1=n+1$ at the FS boundary.
For the radial function $f_i$ of degree $j_i$ in a fixed eigenspace,
the normalized fourth-to-second moment ratio is at most
\[
 16^{j_i}\frac{(2K+3)!}{4^{K+1}((K+1)!)^2}.
\]
Group the full angular multiplicity of degree $i$ into a single spherical
harmonic $Y_i$, as in \eqref{eq:grouped}. For a fixed eigenvalue there
is at most one $j_i$ for each $i$. The angular estimate
\eqref{eq:angular} applies to this entire $Y_i$. Minkowski's inequality
and Cauchy--Schwarz, precisely as in \eqref{eq:combine}, then give the
fourth-moment constant
\[
 \frac{(2K+3)!}{4^{K+1}((K+1)!)^2}
 \left(\sum_i4^{j_i}3^i\right)^2.
\]
By \eqref{eq:fs-joint-cutoff},
\[
 \sum_i4^{j_i}3^i\le4^{K+1}\sum_{i=0}^K\left(\frac34\right)^i
 =4^{K+2}\left(1-\left(\frac34\right)^{K+1}\right).
\]
This proves \eqref{eq:fs-low-moment}. All scaling factors cancel because
$\int V^q\rd\nu_1=1$. Finally,
\[
 \frac{\mu_{K,0}-1}{\theta}
 \ge\frac{(n+K)(n+K+1)-n(n+1)}{2(n+1)}
 =K+\frac{K(K-1)}{2(n+1)}\ge K,
\]
which proves \eqref{eq:fs-angular-gap}.
\end{proof}

\subsection{Truncation around an approximate profile}

For the difficulty of quartic expansions when $2<q<4$, and the use of a perturbative region and a tail region, see \cite[Section~4.1, Lemma~8 and equations~(4.1)--(4.5)]{Frank-Peteranderl=2024}. Here we use the three-range truncation of \cite[Section~2]{Dolbeault etc=2025} around an approximate profile, retaining its Euler--Lagrange residual and keeping the constants uniform in the dimension.

\begin{lemma}\label{lem:fs-robust}
There are $\delta,c,C>0$, depending only on $a$, such that, for
$\phi=\phi_{0,\tau}$ in a fixed sufficiently small parameter ball,
\[
 \phi+h\ge0,\qquad
 h\perp_{\La,\nu_1}\mathcal T_{0,\tau},\qquad
 \|h\|_{q,\nu_1}^2\le\delta
\]
imply
\begin{equation}\label{eq:fs-robust-deficit}
 \mathsf D(\phi+h)\ge\mathsf D(\phi)+c\theta\mathsf E(h)
             -\frac C\theta\|R\|_*^2.
\end{equation}
\end{lemma}

\begin{proof}
Choose a small fixed $\varepsilon>0$. In this proof take
$\varepsilon_0=\varepsilon_1=\varepsilon$ in the three-range
decomposition: thus $\gamma=\varepsilon/2$, $M=1+\varepsilon$,
and take $\varepsilon_2,\sigma_0$ by \eqref{sigma-0} with
$\varepsilon_0=\varepsilon$. In particular,
$\varepsilon_2=(1-3\varepsilon)/4$ and
$\sigma_0=(1-3\varepsilon)/(2q)$ are uniformly positive.
Define the truncations as before, with $\phi$ in place of $V$:
\[
 h_1=\min\{h,\gamma\phi\},\qquad
 h_2=\min\{(h-\gamma\phi)_+,(M-\gamma)\phi\},\qquad
 h_3=(h-M\phi)_+.
\]
By homogeneity, Proposition~\ref{lemma-1} applies to every positive
profile, so \eqref{est-1} holds with these substitutions. Denote its
constant $C_{\gamma,\varepsilon_2,M}$ by $C_{\rm tail}$ in this proof.
It depends only on the fixed $\varepsilon$, uniformly for $2<q\le3$.

Concavity of $x^{2/q}$ at $\int\phi^q\rd\nu_1=1$ yields
\begin{equation}\label{eq:fs-concavity-remainder}
 \mathsf D(\phi+h)-\mathsf D(\phi)-2R(h)
 \ge\mathsf E(h)-\frac2q\int_\CR
       \bigl[(\phi+h)^q-\phi^q-q\phi^{q-1}h\bigr]\rd\nu_1.
\end{equation}
This retains the linear residual and requires no vanishing of
$\int\phi^{q-1}h\rd\nu_1$. The scalar estimate gives, as in
\eqref{eq:deficit-decomposition}, the lower bound
$\varepsilon\theta\mathsf E(h)+I_1+\cdots+I_5$ for the right side,
where all integrals in the following display use $\nu_1$:
\begin{align*}
 I_1={}&(1-\varepsilon\theta)\mathsf E(h_1)
 -(q-1+\varepsilon\theta)\int_\CR\phi^\theta h_1^2
 +\sigma_0\theta\int_\CR\phi^\theta(h_2^2+h_3^2),\\
 I_2={}&(1-\varepsilon\theta)\mathsf E(h_2)
 -\bigl(q-1+(C_{\rm tail}+\sigma_0)\theta\bigr)
                \int_\CR\phi^\theta h_2^2,\\
 I_3={}&(1-\varepsilon\theta)\mathsf E(h_3)
 -\frac2q(1+\varepsilon_2\theta)\int_\CR h_3^q
 -\sigma_0\theta\int_\CR\phi^\theta h_3^2,\\
 I_4={}&2(1-\varepsilon\theta)\langle h_1,h_2\rangle_{\La,\nu_1}
                  -\frac4q\int_\CR\phi^\theta h_1h_2,\\
 I_5={}&2(1-\varepsilon\theta)\langle h_1+h_2,h_3\rangle_{\La,\nu_1}
                  -\frac4q\int_\CR\phi^\theta(h_1+h_2)h_3.
\end{align*}
The factors $2/q\le1$ have only enlarged the coefficients
$\varepsilon$ and $C_{\rm tail}$ in $I_1,I_2$.

First, positivity of $\mathcal A\phi/\phi$ from
Lemma~\ref{lem:fs-profile} and the ground-state identity imply
\[
 \mathsf E(\phi T(h/\phi))
 =\int_\CR\phi^2|\nabla T(h/\phi)|^2\rd\nu_1
 +\int_\CR\frac{\mathcal A\phi}{\phi}\,
                   \phi^2T(h/\phi)^2\rd\nu_1.
\]
Each of the three scalar truncations is $1$-Lipschitz and satisfies
$|T(s)|\le|s|$. Hence
\begin{equation}\label{eq:fs-truncation-energy}
 \mathsf E(h_i)\le\mathsf E(h)\qquad(i=1,2,3).
\end{equation}
On the support of $h_2$ one has $h_1=\gamma\phi$; on the support of
$h_3$ one has $h_1+h_2=M\phi$. Consequently,
\begin{align*}
 I_4={}&2\gamma\theta(1/q-\varepsilon)
            \int_\CR\phi^{q-1}h_2\rd\nu_1
       +2\gamma(1-\varepsilon\theta)R(h_2),\\
 I_5={}&2M\theta(1/q-\varepsilon)
            \int_\CR\phi^{q-1}h_3\rd\nu_1
       +2M(1-\varepsilon\theta)R(h_3).
\end{align*}
The integral terms are nonnegative if $\varepsilon<1/3$.
By \eqref{eq:fs-truncation-energy}, their residual terms together with
$2R(h)$ cost at most $C\|R\|_*\mathsf E(h)^{1/2}$.

For $I_1$, let $p$ be the energy projection of $h_1$ onto
$\mathcal T_{0,\tau}$. Both assertions of Lemma~\ref{lem:fs-forms}
give
\[
 B_\phi(h_1,h_1)\ge c_a\theta\mathsf E(h_1)
                              -C_a\theta\mathsf E(p).
\]
Here the mixed term between $p$ and $h_1-p$ is controlled by
\eqref{eq:fs-tangent-form} and Young's inequality.
Set
\[
 Q_1:=\int_\CR\phi^\theta h_1^2\rd\nu_1,\qquad
 Q_{23}:=\int_\CR\phi^\theta(h_2^2+h_3^2)\rd\nu_1,
 \qquad S:=\{h>\gamma\phi\}.
\]
Orthogonality of $h$ implies
$\langle h_1,w\rangle_{\La,\nu_1}
=-\langle h_2+h_3,w\rangle_{\La,\nu_1}$ for every tangent vector $w$.
Using the whole-unit-tangent fourth-moment estimate in
Lemma~\ref{lem:fs-chart}, H\"older's inequality yields
\begin{align*}
 \mathsf E(p)
 &=\sup_{\substack{w\in\mathcal T_{0,\tau}\\\mathsf E(w)=1}}
            |\langle h_2+h_3,w\rangle_{\La,\nu_1}|^2\\
 &\le C_a\mu_\phi(S)^{1/2}
                 \int_\CR\phi^\theta(h_2+h_3)^2\rd\nu_1\\
 &\le\frac{C_a\sqrt\delta}{\gamma}\sqrt{Q_1Q_{23}}.
\end{align*}
Indeed, $\mu_\phi(S)\le Q_1/\gamma^2$, whereas the final integral
is at most $\min\{\delta,2Q_{23}\}\le\sqrt{2\delta Q_{23}}$.
Also $Q_1\le\|h_1\|_{q,\nu_1}^2\le\mathsf E(h_1)$.
Choose $\varepsilon$ small relative to the form constant $c_a$.
Then
\[
 I_1\ge\theta\left(c'_aQ_1+\sigma_0Q_{23}
               -\frac{C_a\sqrt\delta}{\gamma}\sqrt{Q_1Q_{23}}\right)
 \ge0
\]
for a fixed sufficiently small $\delta$. The lower bound for $\sigma_0$
is uniform in $q$.

For $I_2$, put $Q_{20}=\int_\CR V^\theta h_2^2\rd\nu_1$ and
let $\Pi_K$ be weighted projection onto the direct sum of all eigenspaces
with $\mu\le\mu_{K,0}$. The support of $h_2$ is contained in $S$.
The density estimate \eqref{eq:fs-density-bound} and Chebyshev's
inequality give
\[
 \mu_0(S)\le C_a\mu_\phi(S)^{1/2}
       \le C_a\gamma^{-q/2}\delta^{q/4}.
\]
For each entire low eigenspace, \eqref{eq:fs-low-moment} bounds its squared projection norm by $D_K^{1/2}\mu_0(S)^{1/2}Q_{20}$. Summing over distinct eigenvalues, without a sum over angular multiplicities, gives
\begin{equation}\label{eq:fs-low-projection}
 \int_\CR V^\theta(\Pi_Kh_2)^2\rd\nu_1
 \le C_a(K+1)(K+2)D_K^{1/2}
          \left(\frac\delta{\gamma^2}\right)^{q/8}Q_{20}.
\end{equation}
For fixed $K$, choose $\delta\le\gamma^2$ small enough to make the
coefficient at most $1/2$, uniformly for $2<q\le3$.
All eigenvalues are at least $1$, and the weighted spectral projections
are also energy orthogonal. Retaining the energy of the low modes yields
\begin{align}
 \mathsf E(h_2)-Q_{20}\ge(\mu_{K,0}-1) \left(Q_{20}-\int_\CR V^\theta(\Pi_Kh_2)^2\rd\nu_1\right)\ge\frac{K\theta}{2}Q_{20}. \label{eq:fs-retained-energy}
\end{align}
Using \eqref{eq:fs-form-comparison}, with
$G^2/(n(n+1))=\mathsf E(h_2)-Q_{20}$, gives
\begin{align*}
 I_2\ge{}&\bigl(1-C_*|\tau|-\varepsilon\theta\bigr)\bigl(\mathsf E(h_2)-Q_{20}\bigr)-\theta\bigl(1+\varepsilon+C_{\rm tail}+\sigma_0+C_*|\tau|\bigr)Q_{20}.
\end{align*}
The constant $C_*$ depends only on $a$ and on the fixed $\varepsilon$,
including $C_{\rm tail}$. Choose a fixed integer
\[
 K\ge8\left(2+\varepsilon+C_{\rm tail}
                         +\sup_{2<q\le3}\sigma_0\right),
\]
and then reduce the parameter radius so that $C_*|\tau|\le1/8$.
For $\varepsilon<1/8$, \eqref{eq:fs-retained-energy} proves $I_2\ge0$.

Finally, $\|h_3\|_{q,\nu_1}\le\|h\|_{q,\nu_1}\le1$ and weighted
H\"older give $\int\phi^\theta h_3^2\rd\nu_1\le\|h_3\|_{q,\nu_1}^2$.
The sharp inequality therefore gives
\[
 I_3\ge\frac\theta q(1-q\varepsilon-4\varepsilon_2)
                        \|h_3\|_{q,\nu_1}^2\ge0,
\]
because $4\varepsilon_2=1-3\varepsilon$ and $q\le3$.

The choices can be made in the order: small fixed $\varepsilon$, fixed
large $K$, small fixed profile radius, and small fixed $\delta$ satisfying
both the $I_1$ bound and \eqref{eq:fs-low-projection}. Increase the lower
dimension threshold to that required by Lemma~\ref{lem:fs-lowmom}.
None of these choices depends on $d$. Combining all five estimates in
\eqref{eq:fs-concavity-remainder} leaves
\[
 \mathsf D(\phi+h)\ge\mathsf D(\phi)
      +\varepsilon\theta\mathsf E(h)
      -C\|R\|_*\mathsf E(h)^{1/2}.
\]
Young's inequality absorbs half of the positive quadratic term and
proves \eqref{eq:fs-robust-deficit}.
\end{proof}

\subsection{A local quartic estimate with a uniform radius}
For each fixed $(q,d)$, separating a quartic zero-mode contribution from a coercive transverse contribution yields the asymptotically sharp local quartic estimate of \cite[Propositions~3--4, Definition~5 and equation~(4.18)]{Frank-Peteranderl=2024}. We combine the preceding estimates to control both contributions on a relative neighborhood whose radius is independent of the dimension.
\begin{lemma}\label{lem:fs-local}
There are $\delta_a\in(0,1)$, $\kappa_a>0$, and $d_0(a)$, depending only on $a<0$, such that, for $d\ge d_0(a)$,
\be\label{eq:fs-local-quartic}
 u\ge0,\qquad
 \frac{\dist_{\La,\nu_1}(u,\MR)^2}{\mathsf E(u)}\le\delta_a
 \quad\Longrightarrow\quad
 \mathsf E(u)\mathsf D(u)
 \ge\frac{\kappa_a}{d^3}\dist_{\La,\nu_1}(u,\MR)^4.
\ee
\end{lemma}
\begin{proof}
The assertion is immediate for $u\in\MR$. We therefore assume $u\notin\MR$ and keep $u$ fixed throughout the argument.

\noindent\textit{Step 1. Normalize the nearest optimizer.}
The existence and orthogonality part of Lemma~\ref{lem:nearest-relative} remains valid at $\La=\La_{FS}$: its proof uses only the equation for $V$, positivity of the correlation, and its decay at infinity. It gives a nearest optimizer
\[
 g=c_*V(\cdot-t_*),\qquad c_*>0,
 \qquad
 \langle u-g,g\rangle_{\La,\nu_1}=0.
\]
The lemma applies equally to $\nu_1$, since this measure is a constant multiple of $\nu$. As $\mathsf E(V)=1$, orthogonality gives the exact identities
\[
 \mathsf E(u)
 =c_*^2+\dist_{\La,\nu_1}(u,\MR)^2,\qquad
 c_*^2=\mathsf E(u)-\dist_{\La,\nu_1}(u,\MR)^2.
\]
Define the normalized function
\[
 u_{\rm g}(t,\omega):=c_*^{-1}u(t+t_*,\omega).
\]
Translation is an energy isometry and preserves $\MR$, so $V$ is a nearest optimizer for $u_{\rm g}$ and
\[
 \begin{aligned}
 \mathsf E(u_{\rm g}-V)
 &=\frac{\dist_{\La,\nu_1}(u,\MR)^2}
 {\mathsf E(u)-\dist_{\La,\nu_1}(u,\MR)^2}
 \le\frac{\delta_a}{1-\delta_a},\\
 \mathsf E(u_{\rm g})&=1+\mathsf E(u_{\rm g}-V).
 \end{aligned}
\]
In particular, taking $\delta_a\le1/2$ ensures $1\le\mathsf E(u_{\rm g})\le2$.

\noindent\textit{Step 2. Choose a dimension-independent neighborhood and modulate.}
Fix the parameter and energy radii $r_a,\varepsilon_a$ in Lemma~\ref{lem:fs-chart}, and denote its decomposition constant by $C_{\rm ch}$, which depends only on $a$. Let $\delta>0$ be the truncation threshold in Lemma~\ref{lem:fs-robust}. Decrease the fixed profile radius $t_0(a)>0$ if necessary so that both Lemmas~\ref{lem:fs-profile} and~\ref{lem:fs-robust} apply for $|\tau|\le t_0(a)$. Choose $\delta_a\in(0,1/2]$ sufficiently small that
\[
 \begin{gathered}
 \frac{\delta_a}{1-\delta_a}<\varepsilon_a^2,\qquad
 C_{\rm ch}\sqrt{\frac{\delta_a}{1-\delta_a}}
       \le\min\{r_a/2,1/2,t_0(a)\},\\
 4C_{\rm ch}^2\frac{\delta_a}{1-\delta_a}\le\min\{\delta,1\}.
 \end{gathered}
\]
All quantities used in this choice depend only on $a$.
Lemma~\ref{lem:fs-chart} now yields
\[
 u_{\rm g}=c\phi_{\sigma,\tau}+h,\qquad
 h\perp_{\La,\nu_1}\mathcal T_{\sigma,\tau},
\]
with
\[
 |c-1|+|\sigma|+|\tau|+\sqrt{\mathsf E(h)}
 \le C_{\rm ch}\sqrt{\frac{\delta_a}{1-\delta_a}}.
\]
Thus $1/2\le c\le3/2$ and $|\tau|\le t_0(a)$. Remove the axial shift and this remaining amplitude by setting
\[
 v(t,\omega):=c^{-1}u_{\rm g}(t+\sigma/\sqrt d,\omega),\qquad
 k(t,\omega):=c^{-1}h(t+\sigma/\sqrt d,\omega).
\]
The translation rule for the profiles and their tangent vectors gives
\[
 v=\phi_{0,\tau}+k\ge0,\qquad
 k\perp_{\La,\nu_1}\mathcal T_{0,\tau}.
\]
Moreover, the sharp inequality and our choice of $\delta_a$ imply
\[
 \|k\|_{q,\nu_1}^2\le\mathsf E(k)
 =c^{-2}\mathsf E(h)
 \le4C_{\rm ch}^2\frac{\delta_a}{1-\delta_a}
 \le\min\{\delta,1\}.
\]
These are exactly the smallness, sign, and orthogonality hypotheses needed for Lemma~\ref{lem:fs-robust}. We also have the useful energy bounds
\[
 \frac49\le\mathsf E(v)=c^{-2}\mathsf E(u_{\rm g})\le8.
\]

\noindent\textit{Step 3. Absorb the Euler--Lagrange residual.}
Apply Lemma~\ref{lem:fs-robust} at $\phi_{0,\tau}$. Lemma~\ref{lem:fs-profile} and $\theta\asymp_a d^{-1}$ give constants $A_a,B_a,C_a'>0$, independent of $d$, such that
\[
 \begin{aligned}
 \mathsf D(v)
 &\ge\mathsf D(\phi_{0,\tau})
       +c_a\theta\mathsf E(k)
       -\frac{C_a}{\theta}
          \|\mathcal A\phi_{0,\tau}-\phi_{0,\tau}^{q-1}\|_*^2\\
 &\ge\frac{A_a}{d^3}|\tau|^4
       +\frac{B_a}{d}\mathsf E(k)
       -\frac{C_a'}{d^5}|\tau|^4.
 \end{aligned}
\]
Indeed, squaring the residual bound produces $|\tau|^4/d^6$, while $\theta^{-1}\le C_ad$ loses only one power of $d$. Increase $d_0(a)$ so that $C_a'/d^2\le A_a/2$ for $d\ge d_0(a)$. It follows that
\be\label{eq:fs-local-absorbed}
 \mathsf D(v)\ge\frac{A_a}{2d^3}|\tau|^4
                  +\frac{B_a}{d}\mathsf E(k).
\ee

\noindent\textit{Step 4. Control the fourth power of the distance.}
The energy bound for the first parameter derivative in Lemma~\ref{lem:fs-chart}, applied in the whole direction $\tau$, gives
\[
 \begin{aligned}
 \nm{\phi_{0,\tau}-V}_{\La,\nu_1}
 &=\left\|\int_0^1\sum_{j=1}^d
        \tau_j\partial_{\tau_j}\phi_{0,\upsilon\tau}\rd\upsilon
        \right\|_{\La,\nu_1}
 \le C_a|\tau|.
 \end{aligned}
\]
This uses an operator bound along the segment from $0$ to $\tau$ for the full parameter direction, with no factor depending on the number of angular coordinates. Since $V\in\MR$,
\[
 \dist_{\La,\nu_1}(v,\MR)
 \le\|v-V\|_{\La,\nu_1}
 \le C_a|\tau|+\sqrt{\mathsf E(k)}.
\]
Using $(x+y)^4\le8(x^4+y^4)$ for $x,y\ge0$, we obtain
\[
 \dist_{\La,\nu_1}(v,\MR)^4
 \le C_a\bigl(|\tau|^4+\mathsf E(k)^2\bigr).
\]
Because $\mathsf E(k)\le1$ and $d\ge1$,
\[
 \frac{\mathsf E(k)}d\ge\frac{\mathsf E(k)^2}{d^3}.
\]
Consequently \eqref{eq:fs-local-absorbed} yields a constant $\kappa_a'>0$, depending only on $a$, such that
\[
 \mathsf D(v)\ge\frac{\kappa_a'}{d^3}
                 \dist_{\La,\nu_1}(v,\MR)^4.
\]
Multiplying by $\mathsf E(v)\ge4/9$ proves
\[
 \mathsf E(v)\mathsf D(v)
 \ge\frac{4\kappa_a'}{9d^3}
                 \dist_{\La,\nu_1}(v,\MR)^4.
\]

\noindent\textit{Step 5. Restore the original function.}
By construction,
\[
 u(t,\omega)=c_*c\,
        v(t-t_*-\sigma/\sqrt d,\omega).
\]
Axial translations preserve the energy, deficit, and distance to $\MR$. Multiplication by $c_*c$ multiplies the energy and deficit by $(c_*c)^2$ and the fourth power of the distance by $(c_*c)^4$. Thus the preceding inequality gives \eqref{eq:fs-local-quartic} for $u$, with $\kappa_a=4\kappa_a'/9$. The choices of $\delta_a$, $\kappa_a$, and $d_0(a)$ are independent of the dimension.
\end{proof}

\subsection{An explicit upper bound}
Perturbations along the additional zero modes with second-order corrections are used in \cite[Section~4.3, proof of Proposition~4]{Frank-Peteranderl=2024} to obtain the sharp fixed-parameter local quartic coefficient. We use an explicit positive family to bound the dimensional order. With the normalized energy and deficit of this section, for $|\varepsilon|<1$ set
\[
 v_\varepsilon(t,\omega):=
 V(t)\left(1-\frac{\varepsilon\omega_d}
 {\cosh(\sqrt\La\,t/n)}\right)^{-n}.
\]
This is a positive $H^1(\CR)$ function. Since $\mathsf E(V)=\int_\CR V^q\rd\nu_1=1$, the binomial series and the beta integrals give
\be\label{eq:fs-trial-moments}
\begin{aligned}
 \mathsf E(v_\varepsilon)
 &=\leftidx{_2}F_1\left(n,n+1;\frac d2;\varepsilon^2\right),\\
 \nm{v_\varepsilon}_{q,\nu_1}^q
 &=\leftidx{_2}F_1\left(n+1,n+1;\frac d2;\varepsilon^2\right).
\end{aligned}
\ee
The moment identities in \cite[equation~(4.6)]{Frank-Peteranderl=2024} give, in Pochhammer notation,
\[
 \frac1{|\S^{d-1}|}\int_{\S^{d-1}}\omega_d^{2k}\rd\omega
 =\frac{(1/2)_k}{(d/2)_k},
\]
and
\[
 \frac{\displaystyle\int_\R
 \cosh^{-2(n+1+k)}(\sqrt\La\,t/n)\rd t}
 {\displaystyle\int_\R
 \cosh^{-2(n+1)}(\sqrt\La\,t/n)\rd t}
 =\frac{(n+1)_k}{(n+3/2)_k}.
\]
For the energy identity, one can also differentiate $v_\varepsilon$ directly and use the FS relation
\[
 \frac{\La}{n^2}=\frac{d-1}{2n+1}
\]
to obtain
\[
 \mathcal Av_\varepsilon
 =\left[1-\frac{n^2}{\La}\varepsilon^2
 +\left(\frac{n^2}{\La}-1\right)\varepsilon^2\omega_d^2\right]
 v_\varepsilon^{q-1}.
\]
Multiplying by $v_\varepsilon$ and integrating yields the first identity in \eqref{eq:fs-trial-moments}.

For each fixed $d$ and $a$, the series in \eqref{eq:fs-trial-moments} are analytic in $\varepsilon^2$ near zero. Since $2/q=n/(n+1)$, their quadratic terms cancel in the deficit, and
\be\label{eq:fs-trial-deficit}
 \mathsf D(v_\varepsilon)
 =\frac{2n(n+1)(2n+2-d)}{d^2(d+2)}\varepsilon^4
 +O_{d,a}(\varepsilon^6).
\ee
To compute the distance, write
\[
 v_\varepsilon=V+\varepsilon w+O_{H^1}(\varepsilon^2),\qquad
 w(t,\omega):=\frac{nV(t)\omega_d}{\cosh(\sqrt\La\,t/n)}.
\]
The angular mean of $w$ is zero, so $w$ is orthogonal in the energy inner product to every radial function. Therefore
\[
 \dist_{\La,\nu_1}(V+\varepsilon w,\MR)^2
 =\varepsilon^2\mathsf E(w).
\]
Distance is Lipschitz in the energy norm. Consequently,
\be\label{eq:fs-trial-distance}
 \dist_{\La,\nu_1}(v_\varepsilon,\MR)^2
 =\frac{2n(n+1)(n+2)}{d(2n+3)}\varepsilon^2
 +o_{d,a}(\varepsilon^2).
\ee
Here the coefficient follows either by direct integration or from
\[
 \mathcal Aw=(q-1)V^{q-2}w,\qquad
 \mathsf E(w)=(q-1)\int_\CR V^{q-2}w^2\rd\nu_1
 =\frac{2n(n+1)(n+2)}{d(2n+3)}.
\]
The coefficient in \eqref{eq:fs-trial-deficit} is positive because $2n+2>d$ on the FS curve with $a<0$.

The normalized quotient agrees with the original weighted quotient by the Emden--Fowler identities and cancellation of the common measure factor. Taking $\varepsilon\to0$ at fixed $d$ in \eqref{eq:fs-trial-deficit}--\eqref{eq:fs-trial-distance}, with $\mathsf E(v_\varepsilon)\to1$, gives
\be\label{eq:fs-explicit-upper}
 \widetilde c_{d,a,b_{FS}(a)}
 \le\frac{(2n+2-d)(2n+3)^2}
 {2n(n+1)(n+2)^2(d+2)}.
\ee
Finally, fixing $a<0$ and expanding the parameters in \eqref{eq:fs-parameters} gives
\[
 n=\frac d2-1-2a+O_a(d^{-1}),\qquad
 2n+2-d=-4a+O_a(d^{-1}).
\]
Thus
\be\label{eq:fs-upper-calculation}
 \frac{(2n+2-d)(2n+3)^2}
 {2n(n+1)(n+2)^2(d+2)}
 =\frac{32|a|}{d^3}+O_a(d^{-4}).
\ee
Only the perturbation limit in each fixed dimension was used to obtain \eqref{eq:fs-explicit-upper}; no dimension-uniform Taylor remainder is required.

\subsection{From local quartic coercivity to global stability}
The global quartic stability inequality at fixed parameters is proved in \cite[Corollary~6 and its proof]{Frank-Peteranderl=2024} by a contradiction argument. To retain quantitative control as the dimension varies, we use the energy-controlled path from Section~\ref{sec:global}.
We continue to use $E_\La$, $D_{\La,q}$, and $\rho$ from Section~\ref{sec:global}, now at $\La=\La_{FS}$. All quantities in this subsection use the original cylinder measure $\nu$.

\begin{lemma}\label{lem:fs-energy-path}
Let $d\ge2$, $a<0$, and $b=b_{FS}(a)$. Every nonnegative $u\in H^1(\CR)$ with $\nm{u}_q=1$ can be joined to $V/\nm{V}_q$ by an $H^1$-continuous path $\Gamma:[0,1]\to H^1(\CR)$ such that
\[
\Gamma(\tau)\ge0,\qquad \nm{\Gamma(\tau)}_q=1,\qquad
E_\La(\Gamma(\tau))\le E_\La(u)\quad(0\le\tau\le1).
\]
\end{lemma}
\begin{proof}
The proof of Lemma~\ref{lem:energy-path} also applies at the Felli--Schneider boundary. We verify the two hypotheses in that proof which depend on the parameters. First, $\La=(a_c-a)^2>0$, and
\[
q=2\sqrt{1+\frac{d-1}{(a_c-a)^2}}<2^*.
\]
For $d\ge3$ this follows from $(a_c-a)^2>a_c^2$; for $d=2$ the exponent $q$ is finite and $2^*=\infty$. Thus the inverse map, the energy-controlled interpolation, and compactness of the even, axially decreasing cone used in that lemma remain valid.

Second, every positive finite-energy solution of
\[
\bigl(-\partial_t^2-\Delta_\omega+\La\bigr)v
=\lambda v^{q-1},\qquad \lambda>0,
\]
is a translate and a positive multiple of $V$ at this boundary as well. This is the endpoint rigidity result of \cite{Dolbeault etc=2016}; its application for $a<0$ and $b=b_{FS}(a)$ is also stated explicitly in \cite[Section~9]{Wei-Wu=2026}. The factor $\lambda$ is removed by multiplying $v$ by $\lambda^{1/(q-2)}$. Consequently every cluster point of the inverse iteration in the proof of Lemma~\ref{lem:energy-path} is the same normalized, centered optimizer. The iteration converges in $H^1$, and concatenating its interpolations gives the asserted path. Only the bound by the initial energy is needed; no monotonicity of the energy along each interpolation is asserted.
\end{proof}

\begin{proposition}\label{prop:fs-local-global}
Define
\[
\widetilde c_{\La,q}^{+}
:=\inf_{\substack{u\ge0\\u\notin\MR}}
\frac{E_\La(u)D_{\La,q}(u)}{\dist_\La(u,\MR)^4},
\qquad
K_{\La,q}(\delta)
:=\inf_{\substack{u\ge0,\ u\notin\MR\\\rho(u)\le\delta}}
\frac{E_\La(u)D_{\La,q}(u)}{\dist_\La(u,\MR)^4},
\quad 0<\delta<1.
\]
Then
\be\label{eq:fs-local-global}
\widetilde c_{\La,q}^{+}\ge\delta^2K_{\La,q}(\delta).
\ee
\end{proposition}
\begin{proof}
By homogeneity take $u\ge0$ with $\nm{u}_q=1$. If $\rho(u)\le\delta$, the definition of $K_{\La,q}(\delta)$ gives the stronger estimate without the factor $\delta^2$. Otherwise, Lemma~\ref{lem:fs-energy-path} and continuity of the distance to $\MR$ give a point $w$ on the path with
\[
\rho(w)=\delta,\qquad \nm{w}_q=1,\qquad E_\La(w)\le E_\La(u).
\]
Here $E_\La\ge S_{\La,q}>0$ on the path, so $\rho$ is continuous. Since $\rho(u)\le1$,
\[
\begin{aligned}
\frac{E_\La(u)D_{\La,q}(u)}{\dist_\La(u,\MR)^4}
&=\frac{1-S_{\La,q}/E_\La(u)}{\rho(u)^2}\\
&\ge 1-\frac{S_{\La,q}}{E_\La(u)}
\ge1-\frac{S_{\La,q}}{E_\La(w)}\\
&=\delta^2\frac{E_\La(w)D_{\La,q}(w)}{\dist_\La(w,\MR)^4}
\ge\delta^2K_{\La,q}(\delta).
\end{aligned}
\]
Taking the infimum proves the claim.
\end{proof}

\begin{proposition}\label{prop:fs-sign-reduction}
Let $\widetilde c_{\La,q}$ denote the corresponding infimum over arbitrary real-valued functions. Then
\be\label{eq:fs-sign-reduction}
\widetilde c_{\La,q}\ge
\left[
\bigl(\widetilde c_{\La,q}^{+}\bigr)^{-1/2}
+\bigl\{2(1-2^{2/q-1})\bigr\}^{-1/2}
\right]^{-2}.
\ee
\end{proposition}
\begin{proof}
If $\widetilde c_{\La,q}^{+}=0$, the assertion is immediate. Otherwise, use the normalization and choice of sign from the proof of Proposition~\ref{prop:sign-reduction}. Its mass-splitting argument applies unchanged at $\La=\La_{FS}$ and gives
\[
D_{\La,q}(u)\ge D_{\La,q}(u_+)+\zeta_qE_\La(u_-),
\]
where $\zeta_q=2(1-2^{2/q-1})$ is the coefficient obtained there. For any $0<k\le\widetilde c_{\La,q}^{+}$, put
\[
A_+:=E_\La(u_+),\qquad x:=\dist_\La(u_+,\MR)^2,
\qquad y:=E_\La(u_-).
\]
The chosen sign ensures $A_+>0$. Splitting the energy and applying Cauchy--Schwarz yield
\[
\begin{aligned}
E_\La(u)D_{\La,q}(u)
&\ge(A_++y)\left(k\frac{x^2}{A_+}+\zeta_qy\right)\ge\bigl(\sqrt{k}\,x+\sqrt{\zeta_q}\,y\bigr)^2.
\end{aligned}
\]
The triangle inequality for distance and a second weighted Cauchy--Schwarz inequality give
\[
\begin{aligned}
\dist_\La(u,\MR)^2
&\le(\sqrt{x}+\sqrt{y})^2\le\bigl(\sqrt{k}\,x+\sqrt{\zeta_q}\,y\bigr)
\bigl(k^{-1/2}+\zeta_q^{-1/2}\bigr).
\end{aligned}
\]
Combining the two displays and letting $k\uparrow\widetilde c_{\La,q}^{+}$ proves \eqref{eq:fs-sign-reduction}. The case $u_-=0$ is included by setting $y=0$.
\end{proof}

\begin{proof}[Proof of the lower bound in Theorem~\ref{thm:degenerate}]
Fix $a<0$ and take $b=b_{FS}^{(d)}(a)$. Lemma~\ref{lem:fs-local} supplies $\delta_a\in(0,1)$, $\kappa_a>0$, and $d_0(a)$ such that
\[
K_{\La,q}(\delta_a)\ge\frac{\kappa_a}{d^3},\qquad d\ge d_0(a).
\]
Although that lemma uses a normalized cylinder measure, its quartic quotient is identical to the one above: multiplying the measure by a positive constant multiplies the energy, the sharp deficit, and the squared energy distance by that same constant. The numerator and denominator therefore acquire the same factor. Multiplying a function by a nonzero constant likewise leaves the quotient unchanged.

Proposition~\ref{prop:fs-local-global} gives
\[
\widetilde c_{\La,q}^{+}\ge\frac{\delta_a^2\kappa_a}{d^3}.
\]
For fixed $a$ the Felli--Schneider exponent satisfies
\[
q=2+\frac4d+O_a(d^{-2}),\qquad
\lim_{d\to\infty}2d\bigl(1-2^{2/q-1}\bigr)=4\log2.
\]
We can therefore choose $d_1(a)\ge d_0(a)$ so that
\[
2\bigl(1-2^{2/q-1}\bigr)\ge\frac{\delta_a^2\kappa_a}{d^3}
\qquad(d\ge d_1(a)).
\]
The right-hand side of \eqref{eq:fs-sign-reduction} is increasing in each of its two positive coefficients. Hence
\[
\widetilde c_{\La,q}\ge\frac{\delta_a^2\kappa_a}{4d^3},
\qquad d\ge d_1(a).
\]
The Emden--Fowler transformation preserves the energy norm, the deficit, and the distance to the optimizer manifold, and thus identifies $\widetilde c_{\La,q}$ with $\widetilde c_{d,a,b_{FS}^{(d)}(a)}$.

For each of the finitely many dimensions $2\le d<d_1(a)$, the quartic stability theorem of Frank and Peteranderl \cite[Theorem~1 and Corollary~6]{Frank-Peteranderl=2024} gives
\[
\widetilde c_{d,a,b_{FS}^{(d)}(a)}>0.
\]
Increase $d_1(a)$ to an integer at least $3$, and set
\[
\beta_a^{FS}:=\min\left\{
\frac{\delta_a^2\kappa_a}{4},\quad
\min_{2\le m<d_1(a)}
m^3\widetilde c_{m,a,b_{FS}^{(m)}(a)}
\right\}>0.
\]
This constant depends on $a$ and is independent of $d$. It gives the asserted lower bound in every dimension $d\ge2$.
\end{proof}

\vskip 0.2in
\noindent{\bf Acknowledgments.} 
Juncheng Wei was supported by National R\&D Program of China (Grant No. 2022YFA10 05602), and Hong Kong General Research Fund (No. 14303125) “On Fujita equation in the critical or supercritical regime”. 
Houwang Li was supported by Beijing Institute of Technology Research Fund Program for Young Scholars and National Natural Science Foundation of China (No. 12601182).

\vskip 0.1in
\noindent{\bf Disclosure on AI assistance.} 
The authors used AI-assisted tools, principally ChatGPT/Codex, for brainstorming, literature searches and LaTeX editing.
AI output was not treated as mathematical authority. The authors independently worked and verified all theorem statements, proofs, calculations, constants, citations, and final prose, and take full responsibility for the contents of the paper.

\end{document}